\documentclass[oneside]{amsart}

\usepackage{amsmath, amsthm, amsfonts}
\usepackage{amssymb}
\usepackage[utf8]{inputenc}
\usepackage{latexsym}
\usepackage{verbatim}
\usepackage{url}
\usepackage{textcomp}
\usepackage[all,cmtip]{xy}
\usepackage{color}
\usepackage{hyperref}
\usepackage{tabularx}
\input{xypic}
\usepackage{enumerate}
\usepackage{amssymb}
\usepackage{array}
\usepackage{lscape}

\DeclareMathAlphabet{\mathfr}{U}{euf}{m}{n}

\newtheorem{theorem}{Theorem}[section]

\newtheorem{proposition}[theorem]{Proposition}
\newtheorem{corollary}[theorem]{Corollary}

\newtheorem{lemma}[theorem]{Lemma}

\theoremstyle{remark}
\newtheorem{remark}[theorem]{Remark}

\newcommand{\etalchar}[1]{$^{#1}$}

\newcommand\acc[2]{\ensuremath{{}^{#1}\hskip-0.3ex{#2}}}

\newcommand{\ra}{\rightarrow}

\newcommand{\Q}{\mathbb Q}

\newcommand{\Qbar}{{\overline{\mathbb Q}}}
\newcommand{\Gal}{\mathrm{Gal}}

\newcommand{\Z}{\mathbb Z}

\newcommand{\M}{\mathrm{M }}

\newcommand{\Aut}{\mathrm{Aut}}
\newcommand{\End}{\operatorname{End}}
\newcommand{\Hom}{\operatorname{Hom}}
\newcommand{\Frob}{\operatorname{Frob}}
\newcommand{\Jac}{\operatorname{Jac}}

\newcommand{\p}{\mathfrak{p}}

\newcommand{\Tr}{\operatorname{Tr}}

\newcommand{\cO}{\mathcal{O}}

\newcommand{\Nm}{\operatorname{N}}
\newcommand{\fq}{\mathfrak{q}}
\newcommand{\cB}{\mathcal{B}}
\newcommand{\cA}{\mathcal{A}}
\newcommand{\Dm}{\mathcal{D}}
\newcommand{\Om}{\mathcal{O}}
\newcommand{\Pm}{\mathcal{P}}
\newcommand{\Qm}{\mathcal{Q}}
\newcommand{\cX}{\mathcal{X}}

\newcommand{\Ic}{\mathcal{I}}
\newcommand{\Pc}{\mathcal{P}}
\newcommand{\im}{\operatorname{im}}
\newcommand{\cyc}[1]{{\mathrm{C}_#1}}
\newcommand{\bigslant}[2]{{\raisebox{.2em}{$#1$}\left/\raisebox{-.2em}{$#2$}\right.}}

\numberwithin{equation}{section}

\newcommand{\fm}{\mathfrak{m}}
\newcommand{\fp}{\mathfrak{p}}

\begin{document}
\title[{Endomorphism algebras of $\Qbar$-split genus $2$ Jacobians over $\Q$}]{Endomorphism algebras of geometrically\\ split genus 2 Jacobians over $\Q$}

\author{Francesc Fit\'e}
\address{Departament de Matemàtiques i Informàtica \\
Universitat de Barcelona and Centre de Recerca Matemàtica\\Gran via de les Corts Catalanes, 585\\
08007 Barcelona, Catalonia}
\email{ffite@ub.edu}
\urladdr{https://www.ub.edu/nt/ffite/}

\author{Enric Florit}
\address{Estudis d'Informàtica, Multimèdia i Telecomunicació\\
Universitat Oberta de Catalunya\\Rambla del Poblenou, 156\\
08018 Barcelona, Catalonia}
\email{efloritz@uoc.edu}
\urladdr{https://enricflorit.com/}

\author{Xavier Guitart}
\address{Departament de Matemàtiques i Informàtica\\
Universitat de Barcelona and Centre de Recerca Matemàtica\\Gran via de les Corts Catalanes, 585\\
08007 Barcelona, Catalonia}
\email{xevi.guitart@gmail.com}
\urladdr{https://www.ub.edu/nt/guitart/}

\date{\today}

\date{\today}

\begin{abstract} The main result of \cite{FG20} classifies the 92 geometric endomorphism algebras of geometrically split abelian surfaces defined over $\Q$. We show that~$54$ of them arise as geometric endomorphism algebras of Jacobians of genus~$2$ curves defined over $\Q$, and that the remaining 38 do not. In particular, we exhibit 38 examples of $\Qbar$-isogeny classes of abelian surfaces defined over $\Q$ that do not contain any Jacobian of a genus $2$ curve defined over $\Q$, and for each of the 54 algebras that do arise we exhibit a curve realizing them.
\end{abstract}

\subjclass[2020]{14H40, 11G10, 14K15, 14K22}

\maketitle

\section{Introduction}

Let $A$ be an abelian surface defined over $\Q$. We say that $A$ is \emph{geometrically split} if its base change $A_\Qbar$ to $\Qbar$ is isogenous to the product of two elliptic curves. By the \emph{geometric endomorphism algebra of $A$}, denoted $\cX(A)$, we refer to the ring of endomorphisms of~$A_\Qbar$ tensored with $\Q$. Let $\cA$ denote the set of isomorphism classes of $\cX(A)$, where $A$ runs over the set of geometrically split abelian surfaces over $\Q$. The finiteness of $\cA$ can be deduced from the theory of complex multiplication and Heilbronn's proof of the Gauss class number problem (see \cite[p. 103]{Sha96}).

The main result of \cite{FG20}, relying on previous results of \cite{FG18}, provides an explicit description of $\cA$. In order to state it, we need to introduce some notations. Denote by $\Dm_{1}$ (resp. $\Dm_{2}$) the finite set of absolute values of discriminants of imaginary quadratic fields of class number $1$ (resp. of class number~$2$). We also denote by~$\Dm_{2,2}$ the finite set of absolute values of discriminants of imaginary quadratic fields with class group isomorphic to $\cyc 2 \times \cyc 2$. By the Gauss class number problem, the sets $\Dm_{1}$, $\Dm_{2}$, and $\Dm_{2,2}$ are finite. Using the discriminant bounds of \cite{Wat03}, for example, one can determine these sets, which are as displayed on the table below.
\begin{table}[h]
\begin{center}
\begin{small}
\begin{tabular}{|c|c|}\hline
$\Dm_1$  &$3,4,7,8,11,19,43,67,163$\\ \hline
$\Dm_2$ & $ 15,20,24,35,40,51,52,88,91,115$\\
$ $ & $123,148,187,232,235,267,403,427$\\ \hline
$\Dm_{2,2}$ & $84,120,132,168,195,228,280,312,340,372,408,435$\\
$  $ &$483,520,532,555,595,627,708,715,760,795,1012,1435$\\\hline
\end{tabular}
\end{small}
\end{center}
\end{table}

Note that $\Dm_1$, $\Dm_2$, and $\Dm_{2,2}$ have cardinalities $9$, $18$, and $24$ respectively. Then the main result of \cite{FG20} is the following.

\begin{theorem}[\cite{FG20}]\label{theorem: FG20}
The set $\cA$ consists of:\vspace{0,2cm}
\begin{enumerate}[i)]
\item $\Q\times \Q$; $\Q\times \Q(\sqrt{-d})$ for $d\in \Dm_1$; $\Q(\sqrt{-d})\times \Q(\sqrt{-d'})$ for distinct $d,d'\in \Dm_1$;\vspace{0,2cm}
\item $\M_2(\Q)$; $\M_2(\Q(\sqrt{-d}))$ for $d\in \Dm_1 \cup \Dm_2 \cup \Dm_{2,2}^S$, where  
\begin{equation}\label{equation: DNakamura}
\Dm_{2,2}^S:=\Dm_{2,2}\setminus \{195, 312, 340, 555, 715, 760\}\,.
\end{equation}
\end{enumerate}
In particular, the set $\cA$ has cardinality 92.
\end{theorem}

The challenging part of the proof of the above result is that involving quadratic imaginary fields of class number $4$. 
In this case, the proof of the above theorem takes place in two steps. Suppose that $A$ is an abelian surface for which $\cX(A)\simeq \M_2(\Q(\sqrt{-d}))$ for $d\in \Dm_{2,2}$. First, one establishes that $A$ occurs as a factor of the restriction of scalars of a Gross $\Q$-curve from the Hilbert class field of $\Q(\sqrt{-d})$ down to $\Q$. Second, one determines the isogeny decomposition of any such restriction of scalars using a method due to Nakamura \cite{Nak04}. Then the set $\Dm_{2,2}^S$ is obtained as a corollary of this computation. Note that a consequence of the first step in the proof is that for such an $A$ the minimal field over which all endomorphisms of $A$ are defined is the Hilbert class field of $\Q(\sqrt{-d})$, a result that will again play a key role in the present work.   

In this article we consider the question (see \cite[Question 4.4]{FG20}) of determining the subset $\cB\subseteq \cA$ of geometric endomorphism algebras of Jacobians of genus~2 curves defined over $\Q$. We obtain a complete answer to this question.

\begin{theorem}\label{theorem: main}
The set $\cB$ consists of:\vspace{0,2cm}
\begin{enumerate}[i)]
\item $\Q\times \Q$; $\Q\times \Q(\sqrt{-d})$ for $d\in \Dm_1$; $\Q(\sqrt{-d})\times \Q(\sqrt{-d'})$ for
$$
(d,d')\in \{(4,3)\,,(7,4)\,,(8,4)\}\,.
$$
\item $\M_2(\Q)$, $\M_2(\Q(\sqrt{-d}))$ for $d\in \Dm_1 \cup \Dm_2 \cup \Dm_{2,2}^J$, where 
\begin{equation}\label{equation: DNarbonneRitzenthaler}
\Dm_{2,2}^J :=\Dm_{2,2}^S\setminus \{ 280,  483,  595, 627, 795\}\,.
\end{equation}
\end{enumerate}
In particular, the set $\cB$ has cardinality $54$.
\end{theorem}

The interesting cases of the proof of the above result are: a) that involving the product of two distinct quadratic imaginary fields; b) that involving a quadratic imaginary field of class number $4$. The remaining cases are handled in \S\ref{section: basic cases}, a section included for the sake of completeness, but whose results are well known to the experts. We do not claim any originality in this section, except maybe for the precise form of Proposition~\ref{proposition: classnumber2}. 
This result provides a Galois group theoretic sufficiency criterion for applying the explicit gluing method of Howe, Leprévost, and Poonen \cite{HLP00} to a CM elliptic curve over a quadratic field and its Galois conjugate along the $2$-torsion. It allowed us to find some curves attached to quadratic imaginary fields of class number 2 that we were not able to locate anywhere else in the literature.

We effectuate the proof of Theorem \ref{theorem: main} in the aforementioned case a) in \S\ref{section: nonisogenousprod}. The proof relies on a criterion by Frey and Kani \cite{FK91} that recasts the possibility of gluing two elliptic curves $E$ and $E'$ into a genus~$2$ curve in terms of the existence of an isomorphism of Galois modules $E[N]\simeq E'[N]$, for some $N\geq 2$, that is moreover an anti-isometry for the Weil pairings. We exploit the idea that, due to the presence of a large proportion of 0 Frobenius traces, only under exceptional circumstances may one find CM elliptic curves $E$ and $E'$ nonisogenous over $\Qbar$ such that $E[N]$ and $ E'[N]$ are isomorphic as Galois modules for some $N\geq 2$. In fact, to narrow down the possibilities for $(d,d')$ to one of $(4,3)$, $(7,4)$, or $(8,4)$ we make no use of the anti-isometry requirement. This idea is made precise via a Chebotaryov style type of argument.

The proof of Theorem \ref{theorem: main} in the aforementioned case b) takes place in \S\ref{section: squares}. Suppose that $C$ is a genus $2$ curve for which $\cX(\Jac(C))\simeq \M_2(\Q(\sqrt{-d}))$ for $d\in \Dm_{2,2}$. Combining the already mentioned result from \cite{FG20} that the minimal field of definition over which all the endomorphisms of $\Jac(C)$ are defined is the Hilbert class field of $\Q(\sqrt{-d})$ with \cite{JKPRST18} and \cite{Kan11}, we infer that $\Jac(C)$ is $\Qbar$-\emph{isomorphic} to the product of two elliptic curves with complex multiplication by the \emph{maximal} order of $\Q(\sqrt{-d})$. The necessity for~$d$ to belong to $\Dm_{2,2}^J$ then is deduced from the recent work by Narbonne \cite{Nar22}. It then remains to show that the algebras $\M_2(\Q(\sqrt{-d}))$ for every $d\in \Dm_{2,2}^J$ actually occur as geometric endomorphism algebras of genus 2 curves defined over $\Q$. For every $d\in \Dm_{2,2}^J$, a heuristic candidate curve was communicated to us either by Elkies or by Narbonne--Ritzenthaler. Finally, we certify the correctness of these curves using the work of Kumar \cite{Kum15}.

If follows from \cite{Wei57} and \cite[Thm. 1]{Kan97} that any $\Qbar$-isogeny class of abelian surfaces defined over~$\Qbar$ contains the Jacobian of a curve defined over $\Qbar$. In contrast, as a byproduct of Theorems \ref{theorem: FG20} and \ref{theorem: main} one has that not every $\Qbar$-isogeny class of abelian surfaces defined over $\Q$ contains the Jacobian of a curve defined over $\Q$. More precisely, we have the following corollary.

\begin{corollary}
There exist 38 distinct $\Qbar$-isogeny classes of abelian surfaces containing an abelian surface defined over $\Q$ but no Jacobian of a curve defined over $\Q$.  
\end{corollary}

Indeed, by Theorems  \ref{theorem: FG20} and \ref{theorem: main} the set $\cA \setminus \cB$ has cardinality 38.  Also, every abelian surface over $\Q$ having geometric endomorphism algebra an element of $\cA\setminus \cB$  is geometrically isogenous to a product of CM elliptic curves, and hence  each element of $\cA\setminus \cB$ corresponds to a unique $\Qbar$-isogeny class of abelian surfaces with the stated property. Note that these 38 isogeny classes split in two cases. First we have those 5 of the form $E^2$ where $\mathcal{X}(E)$ is $\Q(\sqrt{-d})$ for $d\in\{280,483,595,627,795\}$, in which case a representative abelian surface over $\Q$ was obtained in \cite{FG20} as a factor of the restriction of scalars of a suitable model for $E$. Second, those 33 of the form $E_1\times E_2$  with $\chi(E_i)\simeq \Q(\sqrt{-d_i})$ where $d_i\in \mathcal{D}_1$,  $d_1> d_2$, and $(d_1,d_2)\not\in \{(4,3), (7,4), (8,4)\}$, in which case $E_1$ and $E_2$ can be naturally chosen to be defined over $\Q$.

Some of the computations in this article are done with the help of mathematical software SageMath \cite{sage} and MAGMA \cite{BCP97}. All code files that we have used can be found at the repository \cite{FFG22}.

\subsection*{Notation and terminology} Throughout the paper $\Qbar$ will denote a fixed algebraic closure of $\Q$. All number fields considered are assumed to be contained in $\Qbar$. For a number field $F$, let $\cO_F$ denote its ring of integers and let $G_F$ be the absolute Galois group $\Gal(\Qbar/F)$. For an abelian variety $A$ defined over $F$ and an integer $N\geq 2$, let $A[N]$ denote the set of $\Qbar$-points of the $N$-torsion group scheme of $A$. For a prime $\ell$, let 
$V_\ell(A):= \Q_\ell\otimes \varprojlim_n A[\ell^n]$ be the rational $\ell$-adic Tate module of~$A$, and denote by $\varrho_A:=\varrho_{A,\ell}$ the $\ell$-adic representation encoding the action of $G_F$ on~$V_\ell(A)$. For a nonzero integral prime~$\p$ of $\cO_F$ of good reduction for $A$, define~$a_\p(A)$ as the trace $\Tr(\varrho_{A}(\Frob_\p))$, where $\Frob_\p$ is a Frobenius element at~$\p$. We will denote by $\overline\varrho_{A,N}$ the representation given by the action of~$G_F$ on~$A[N]$. We will denote by $F(A[N])$ the field obtained by adjoining the coordinates of the points of $A[N]$ to $F$ (equivalently, $F(A[N])$ is the field cut out by the representation $\overline\varrho_{A,N}$). We will denote by $F(x(A[N]))$ the field obtained by adjoining the $x$-coordinates of the points of $A[N]$ to $F$. The geometric endomorphism algebra of~$A$ will be denoted by $\cX(A)$. If $C$ is a curve, we will simply write $\cX(C)$ to denote $\cX(\Jac(C))$. If $B$ is another abelian variety defined over $F$, then we write $A\sim B$ (resp. $A\simeq B$) if $A$ and $B$ are isogenous (resp. isomorphic) over $F$. Similarly, $\Hom(A,B)$ denotes homomorphisms from $A$ to $B$ defined over $F$. We say that an elliptic curve $E$ has complex multiplication (CM) by a quadratic imaginary field $M$ if it has CM by some order~$\cO$ in $M$. We use this phrasing when the specific choice of~$\cO$ is irrelevant for the argument involved.

\subsection*{Acknowledgements} We are indebted to several colleagues for contributions that made possible this work. Thanks to N. Elkies for providing a key argument to complete the proof of Proposition \ref{proposition: just one prime} and for supplying some extraordinary genus~2 curves that motivated the start of this article. Thanks to A. Sutherland for sharing a data base of pairs of isomorphic torsion Galois modules of elliptic curves (with overwhelming evidence that Theorem \ref{theorem: three pairs} should be true) and for granting us access to the MIT server Legendre, where some of the computations of this article were carried out. Thanks to F. Narbonne for sharing with us the work \cite{Nar22}, essential in the proof of Proposition \ref{proposition: laststep}. Thanks to F. Narbonne and C. Ritzenthaler for supplying some extraordinary genus 2 curves that facilitated the completion of this article. Thanks to E. Costa, E. Howe, S. Schiavone, and J. Sijsling for helpful conversations, and to the anonymous referee for corrections and suggestions that substantially improved the content and the exposition of the article.  Guitart expresses his gratitude to the Massachusetts Institute of Technology for its warm hospitality during a visit in January 2020. All three authors were supported by the grant PID2022-137605NB-I00.  Fit\'e was financially supported by the Ram\'on y Cajal fellowship RYC-2019-027378-I and the Simons Foundation grant 550033. Florit was partially supported by the Spanish Ministry of Universities (FPU20/05059). Fité and Guitart acknowledge support from the Mar\'ia de Maeztu Program of excellence CEX2020-001084-M.

\subsection*{AI disclosure} No AI tools were used in the mathematical content of this work, and no AI tools were used at any stage in the original manuscript. During the revision process, AI tools were used to check the grammar of a few sentences.

\section{The basic cases}\label{section: basic cases}

In this section we treat those cases of geometric endomorphism algebras involving at most one quadratic imaginary field, which we further assume of class number $\leq 2$. Along the way, we present a gluing construction of elliptic curves that will be also used in Section \ref{section: nonisogenousprod}.

\subsection{Class number 1}

Examples of genus 2 curves $C$ defined over $\Q$ for which $\cX(C)\simeq \Q\times \Q$ or $\M_2(\Q)$ are well known (see for example the rows for $N(G_{3,3})$ and~$E_1$ of \cite[Table 11]{FKRS12}, respectively). 
The next proposition provides examples of genus 2 curves $C$ defined over $\Q$ such that $\cX(C)\simeq \Q\times M$, where $M$ is an imaginary quadratic field of class number $1$.

\begin{proposition}\label{proposition: QxM}
For each $d\in \Dm_1$, consider the pair $(a,b)\in \Z^2$ given by Table \ref{table: ab}. If $C$ denotes the curve given by $C\colon y^2=x^6+ax^2+b$, then $\cX(C)\simeq \Q \times \Q(\sqrt{-d})$.
\begin{table}[h]
\begin{center}
\begin{small}
\begin{tabular}{|rr|rr|rr|}\hline
$(a,b)$ & $d$ & $(a,b)$ & $d$ & $(a,b)$ & $d$\\\hline
$(-15,22)$ & $3$ & $( -30, -56) $ & $8$ & $(-3440, 77658) $ & $43$\\
$(-11,14)$ & $4$ & $(-264, 1694) $ & $11$ & $(-29480, 1948226)$ & $67$\\
$(-35,98)$ & $7$ & $(-152, 722)$ & $19$ & $(-8697680, 9873093538)$ & $163$\\\hline
\end{tabular}
\caption{Coefficients for genus 2 curves $C$ with $\cX(C)\simeq \Q\times \Q(\sqrt{-d})$.}\label{table: ab}
\end{small}
\end{center}
\end{table}
\end{proposition}

\begin{proof}
The curve $C$ admits obvious degree 2 maps to the elliptic curves
$$
E\colon y^2=x^3+ax+b\,, \qquad E'\colon y^2=x(x^3+ax+b)\,.
$$
A straightforward computation shows that, for each pair $(a,b)$ of Table \ref{table: ab}, the $j$-invariant of $E$ is that of an elliptic curve with CM by $\Q(\sqrt{-d})$ (cf. \cite[page 183]{Sil94}), while the $j$-invariant of $E'$ is not integral and thus $E'$ does not have CM.
\end{proof}

To conclude the exploration of cases involving fields of class number 1, we recall that the literature contains examples of genus 2 curves defined over $\Q$ such that $\cX(C)\simeq \M_2(\Q(\sqrt{-d}))$ for every $d\in \Dm_1$. For the sake of completeness, explicit examples of all of them are provided in \S\ref{section: examples}. The curves $y^2=x^6+1$ and $y^2=x^5-x$ are well known to correspond to the cases $d=3,8$. For $d=4,7$ we borrow examples from \cite{Car01}, and for $d=11,19,43,67,163$ the curves are taken from \cite{GHR19}.

\subsection{Gluing of elliptic curves}\label{section: gluingec} Although the results recalled in this section hold in greater generality, we will state them relative to a number field $F$. This will be enough for our purposes. 

The following result combines results of \cite{FK91} and \cite{Kan19}. More precisely, \cite[\S1]{FK91} treats the particular case of elliptic subcovers, and \cite[Prop. 10]{Kan19} extends the result to the more general case of principally polarized split abelian surfaces satisfying condition \cite[(2)]{Kan19}. For split Jacobians this additional condition is satisfied by \cite[Rmk. 11 (a)]{Kan19}. This result  will be used in \S\ref{section: squares} as a fundamental tool to find an upper bound on the set $\cB$ of Theorem \ref{theorem: main}. 

\begin{proposition}[Frey--Kani]\label{proposition: FreyKani}
Let $C$ be a genus 2 curve defined over $F$ such that $\Jac(C)$ decomposes up to isogeny. Then there exist elliptic curves $E$ and $E'$ defined over $F$, a natural number $N\geq 2$ and an isomorphism $\psi\colon E[N]\rightarrow E'[N]$ of $(\Z/N\Z)[G_F]$-modules which is also an anti-isometry with respect to the Weil pairings such that there exists an isogeny
$$
\pi\colon E\times E' \longrightarrow \Jac(C)
$$
whose kernel $\Gamma$ is the graph of $\psi$, that is, the group scheme whose $\Qbar$-points are of the form $(P,\psi(P))$ with $ P\in E[N](\Qbar)$. Moreover, if we equip $\Jac(C)$ with its canonical polarization coming from the theta divisor of $C$ and $(E\times E')/{\Gamma}$ with the polarization induced in the quotient by $N(E\times\{0\} + \{0\}\times E')$, then $\pi$ induces an isomorphism
\[
\varphi\colon \bigslant{(E\times E')}{\Gamma} \simeq \Jac(C)
\]
of principally polarized abelian surfaces.

\end{proposition}

For $N=2$, there is an explicit converse of the above result due to Howe, Lepr\'evost and Poonen (see \cite[Prop. 3 and 4]{HLP00}). It will be used in \S\ref{section: squareclassnumber2} and \S\ref{section: reapairs}, in order to realize some of the elements in $\cB$ as geometric endomorphism algebras of Jacobians over $\Q$. Although it will not be used in this article, we note that there is also an explicit converse to the above result for $N=3$ (see \cite[Alg. 5.4, Thm. 5.5]{BHLS15}).

\begin{proposition}[\cite{HLP00}]\label{proposition: gluing}
Let $E$ and $E'$ be elliptic curves defined over $F$ such that there exists an isomorphism $\psi\colon E[2]\rightarrow E'[2]$ of $(\Z/2\Z)[G_F]$-modules which is not the restriction to $E[2]$ of an isomorphism $E_\Qbar \rightarrow E'_\Qbar$. Let $\Gamma$ be the graph of $\psi$, and equip $(E\times E')\slash\Gamma$ with the principal polarization described in Proposition \ref{proposition: FreyKani}. Then:
\begin{enumerate}[i)]
\item There exist a genus $2$ curve~$C$ defined over $F$ and an isomorphism 
\begin{equation}\label{equation: isojac}
\Jac(C)\simeq \bigslant{(E\times E')}{\Gamma}
\end{equation}
of principally polarized abelian surfaces.
\item Suppose that $E$ and $E'$ are given by equations $y^2=f(x)$ and $y^2=g(x)$, where $f,g\in F[x]$ are monic and cubic. Let $\alpha_1,\alpha_2,\alpha_3$ (resp. $\beta_1,\beta_2,\beta_3$) be the roots of $f$ (resp. $g$) numbered such that $\psi((\alpha_i,0))=(\beta_i,0)$. Set $\alpha_{ij}=\alpha_i-\alpha_j$ and $\beta_{ij}=\beta_i-\beta_j$ and define
$$
\begin{array}{rr}
a_1=\alpha^2_{32}/\beta_{32}+\alpha^2_{21}/\beta_{21}+\alpha_{13}^2/\beta_{13}, & a_2=\alpha_1\beta_{32}+\alpha_2\beta_{13}+\alpha_3\beta_{21}, \\
b_1=\beta^2_{32}/\alpha_{32}+\beta^2_{21}/\alpha_{21}+\beta_{13}^2/\alpha_{13}, & b_2=\beta_1\alpha_{32}+\beta_2\alpha_{13}+\beta_3\alpha_{21}. 
\end{array}
$$
Let $A=\Delta_ga_1/a_2$ and $B=\Delta_f b_1/b_2$, where $\Delta_f$ and $\Delta_g$ are the discriminants of $f$ and $g$, respectively. Then a model for $C$ satisfying \eqref{equation: isojac} is given by

\begin{align*}
y^2=-(A\alpha_{21}\alpha_{13}x^2+B\beta_{21}\beta_{13})(A\alpha_{32}\alpha_{21}x^2+B\beta_{32}\beta_{21})(A\alpha_{13}\alpha_{32}x^2+B\beta_{13}\beta_{32})\,.
\end{align*}

\end{enumerate}   
\end{proposition} 

\subsection{Class number 2}\label{section: squareclassnumber2} Examples of genus 2 curves $C$ defined over $\Q$ such that $\cX(C)\simeq \M_2(\Q(\sqrt{-d}))$, with $d\in \Dm_2\setminus \{35,91,115,187,235,403,427\}$, can be found in the sources \cite{Car01} and \cite{GHR19} (for the reader's convenience we reproduce them in \S\ref{section: examples}).
We will find curves realizing the remaining discriminants by making use of the proposition and remark below. 

\begin{proposition}\label{proposition: classnumber2} 
Let $M$ be an imaginary quadratic field of class number $2$. Let $H_2$ denote the ray class field modulo~$2$ of $M$ and let $H^+$ denote the maximal real subfield of the Hilbert class field $H$ of $M$. Set $G=\Gal(H_2/\Q)$ and $G^+=\Gal(H_2/H^+)$, and let $Z$ denote the centralizer of $G^+$ in $G$. Suppose that
\begin{equation}\label{equation: intersection}
(G\setminus G^+) \cap Z
\end{equation} 
contains an element of order $2$. Then there exists a genus $2$ curve $C$ defined over~$\Q$ such that $\cX(C)\simeq \M_2(M)$.
\end{proposition}

\begin{proof}
  Let $E$ be an elliptic curve with CM by the ring of integers of $M$. Since $M$ has class number 2, by the theory of complex multiplication, as explained in \cite[\S5.1]{Gross}, the field $\Q(j(E))$ is real quadratic and $H^+=\Q(j(E)).$ Since $j(E)\not = 0, 1728$, by \cite[Thm. 4.3, Thm. 5.6]{Sil94} we have that
$$
H_2=H(x(E[2]))=H(E[2])\,,
$$
and thus the action of $G_{H^+}$ on $E[2]$ factors through $G^+$. 
Let $\sigma$ be an element of order $2$ in the intersection \eqref{equation: intersection}. We will apply part $i)$ of Proposition \ref{proposition: gluing} to the pair of elliptic curves $E$ and $\acc \sigma E$ defined over the field $H^+=\Q(j(E))$.

We claim that the map $\psi_\sigma\colon E[2] \rightarrow \acc \sigma E[2]$ defined as $\psi_\sigma(P)=\acc\sigma P$  is an isomorphism of $(\Z/2\Z)[G^+]$-Galois modules. Indeed, the $\Z/2\Z$-linearity of $\psi_\sigma$ is a direct consequence of the fact that $\psi_\sigma$ is a bijection that maps $0$ to $0$; its $G^+$-equivariance is a consequence of the fact that $\sigma$ commutes with every element in $G^+$. Since $\sigma \not \in G^+$, we have that $\sigma$ restricts to the nontrivial automorphism of $H^+$ and hence $j(E)\not =j(\acc \sigma E)$. This implies that $\psi_\sigma$ does not come from an isomorphism $E_\Qbar \rightarrow \acc\sigma E_\Qbar$. Let $\Gamma_\sigma$ be the graph of $\psi_\sigma$. By Proposition \ref{proposition: gluing}, there exists a genus $2$ curve $\tilde C$ defined over~$H^+$ and an isomorphism
$$
\varphi_\sigma\colon \Jac(\tilde C)\simeq \bigslant{(E\times \acc \sigma E)}{\Gamma_\sigma}\,,
$$
of principally polarized abelian surfaces.  We next show that $\tilde C$ can be descended to $\Q$. Let $s\colon \acc\sigma E\times  E \ra  E \times \acc\sigma E$ be the isomorphism that swaps $E$ and $\acc\sigma E$. Note that $s$ takes the product polarization of $\acc\sigma E \times E$ into that of $E \times \acc\sigma E$. As $\sigma$ has order 2, we have that $s(\acc\sigma\,\Gamma_\sigma)=\Gamma_\sigma$ and therefore the map induced by $s$
\[
\xi_\sigma\colon \bigslant{(\acc \sigma E\times E)}{\acc\sigma \,\Gamma_\sigma}{\longrightarrow} \bigslant{(E\times \acc \sigma E)}{\Gamma_\sigma}
  \]
is an isomorphism of principally polarized abelian surfaces. Therefore, 
the composition
\begin{equation*}
\nu_\sigma\colon \acc \sigma \Jac(\tilde C)\xrightarrow{\acc \sigma \varphi_\sigma} \bigslant{(\acc \sigma E\times E)}{\acc\sigma \,\Gamma_\sigma}\xrightarrow{\xi_\sigma} \bigslant{(E\times \acc \sigma E)}{\Gamma_\sigma} \xrightarrow{\varphi_\sigma^{-1}} \Jac(\tilde C)
\end{equation*}
 is an isomorphism of principally polarized abelian surfaces. A trivial computation shows that the isomorphism~$\nu_\sigma$ verifies that $\nu_\sigma\circ\acc \sigma \nu_\sigma=\mathrm {id}$, that is, it provides Weil descent data for $\Jac(\tilde C)$ from~$H^+$ to~$\Q$. 

It suffices to show that this descent data produces descent data for $\tilde C$. Any isomorphism 
$\xi$ between curves induces by transport of structure an isomorphism $\Jac(\xi)$ between their Jacobians.
By Torelli's theorem (see \cite[Thm. 12.1 and Cor. 12.2]{Milne} for the general statement and \cite[Thm. 1 in Appendice]{Lau01} for a refinement for hyperelliptic curves), there exists a \emph{unique} isomorphism $\lambda_\sigma:\acc \sigma {\tilde C} \rightarrow \tilde C$ such that $\Jac(\lambda_\sigma)=\nu_\sigma$. By functoriality, we have
$$
\Jac(\lambda_\sigma\circ \acc\sigma  \lambda_\sigma)=\Jac(\lambda_\sigma)\circ \Jac(\acc\sigma  \lambda_\sigma)=\nu_\sigma\circ\acc \sigma \nu_\sigma=\mathrm {id}\,,
$$
which by the unicity part of Torelli's theorem implies that $\lambda_\sigma\circ \acc\sigma  \lambda_\sigma=\mathrm{id}$.
By \cite[Thm. 1]{Wei56} there exists $C$ defined over $\Q$ such that $C_{H^+}\simeq \tilde C$, and this completes the proof of the proposition.
\end{proof}

\begin{remark}\label{remark: classnumber2}
Resume the notations of Proposition \ref{proposition: classnumber2}. We have verified with MAGMA that, for every $d\in \Dm_2\setminus \{20,52,148\}$ there is an element $\sigma$ of order~$2$ in $(G\setminus G^+)\cap Z$. To obtain a curve $C$ as in the statement of the proposition we proceed by making effective each step of its proof. Indeed, we choose a root $j$ of the Hilbert class polynomial attached to the discriminant of $M$ and an elliptic curve $E$ such that $j(E)=j$. We choose $\sigma\in (G\setminus G^+)\cap Z$ of order~$2$, and apply part $ii)$ of Proposition \ref{proposition: gluing} to $E$ and $\acc \sigma E$ obtaining a genus 2 curve $\tilde C$ defined over $\Q(j(E))$.
We then compute the Cardona-Quer-Nart-Pujol\`as invariants $(g_1,g_2,g_3)$ of $\tilde C$ (as defined in \cite{CNP05}, \cite{CQ05}), which as predicted by Proposition \ref{proposition: classnumber2} are rational numbers. From $(g_1,g_2,g_3)$, we construct a genus $2$ curve $C$ defined over $\Q$ using MAGMA.  In \S\ref{section: examples}, we present the results of this computation for $d\in \{35,91,115,187,235,403,427\}$. The code used for this computation can be found in the file \texttt{SquareClassNumberTwo.m} of \cite{FFG22}.
\end{remark}

\section{Products of nonisogenous CM elliptic curves}\label{section: nonisogenousprod}

In \S\ref{section: upperbound} we show that if $C$ is a genus $2$ curve defined over $\Q$ whose Jacobian is geometrically isogenous to the product of two nonisogenous CM elliptic curves, then $\cX(C)$ is isomorphic to
\begin{equation}\label{equation: 3pairs}
\Q(\sqrt{-d})\times \Q(\sqrt{-d'})\,,\quad \text{where }
(d,d')\in \{(4,3)\,,(7,4)\,,(8,4)\}\,.
\end{equation}
In \S\ref{section: reapairs} we construct genus 2 curves over $\Q$ realizing each of the above three pairs.

\subsection{The upper bound}\label{section: upperbound} In this section, $M$ and $M'$ denote nonisomorphic imaginary quadratic fields of class number $1$ and discriminants $-d$ and $-d'$, respectively. For $z\in M$ we denote by $\overline{z}$ its complex conjugate.

\begin{lemma}\label{lemma: appm1}
Assume that $d\not =3,4$, and let $E/\Q$ and~$E'/\Q$ be elliptic curves with CM by $M$ and $M'$, respectively. Suppose that $p\geq 5$ is a prime of good reduction for $E$ and $E'$, split in $M$ and inert in $M'$. Then
$$
a_p(E)=\pm \Tr_{M/\Q}(\pi)\quad \text{and}\quad a_p(E')=0\,,
$$
where $\pi$ is any element in $\cO_M$ such that $p=\pi\overline \pi$.
\end{lemma}

\begin{proof}
Since $p$ is supersingular for $E'$ and $p\geq 5$, we have that $a_p(E')=0$. By \cite[Thm. 14.16]{Cox89}, we have that there exists $\pi\in \cO_M$ such that $a_p(E)=\Tr_{M/\Q}(\pi)$ and $p=\pi\overline \pi$. Since $\cO_M^\times=\{\pm 1\}$ by assumption, the element $\pi$ is uniquely determined up to sign.
\end{proof}

\begin{proposition}\label{proposition: just one prime}
Assume that $d\not =3,4$, and let $N\geq 2$ and $p_0\geq 5$ be distinct rational primes. Suppose that:
\begin{enumerate}[i)]
\item $p_0$ is inert in $M'/\Q$.
\item $p_0$ is split in $M/\Q$ and $N$ does not divide $\Tr_{M/\Q}(\pi_0)$, where $\pi_0 \in \cO_M$ is such that $p_0=\pi_0 \overline \pi_0$.
\end{enumerate} 
Then for every pair of elliptic curves $E$ and $E'$ defined over $\Q$ with CM by $M$ and~$M'$ respectively, we have that $E[N]$ and $E'[N]$ are not isomorphic as $(\Z/N\Z)[G_\Q]$-modules.  
\end{proposition}

\begin{proof}
Let $a$ denote $\Tr_{M/\Q}(\pi_0)$ and let $P_{a,N}$ be the set of primes $p$ such that:
\begin{enumerate}[i)]
\item $p$ is inert in $M'/\Q$.
\item $p$ is split in $M/\Q$ and 
$$
\Tr_{M/\Q}(\pi)\equiv a \pmod N\,,
$$
for some $\pi\in \cO_M$ such that $p=\pi \overline \pi$.
\end{enumerate} 
We claim that $P_{a,N}$ is infinite. It suffices to show that there exist infinitely many primes $\pi\in \cO_M$ of residue degree $1$ such that 
\begin{equation}\label{equation: congruence}
  \pi \equiv \pi_0 \pmod {\mathfrak{m}}\,,
\end{equation}
where $\mathfrak{m}$ is the ideal generated by $ \mathrm{lcm}(N,d')$ in $\cO_M$. 
Indeed, for such a $\pi$, we claim that $p=\Nm_{M/\Q}(\pi)$ belongs to $P_{a,N}$. On the one hand, we have
$$
p=\Nm_{M/\Q}(\pi)\equiv \Nm_{M/\Q}(\pi_0)= p_0 \pmod {d'},
$$
and thus $p$ is inert in $M'/\Q$. On the other hand, we have
$$
\Tr_{M/\Q}(\pi)\equiv \Tr_{M/\Q}(\pi_0)=a \pmod N\,.
$$ 
Note that $\pi_0$ is coprime to $d'$, since $p_0$ is inert in $M'/\Q$, and it is coprime to $N$ by hypothesis. Then, the infiniteness of the set of primes $\pi$ satisfying \eqref{equation: congruence} follows from the analogue of Dirichlet's theorem over the ring of integers $\cO_M$. More precisely, let $C_\fm$ be the ray class group modulo $\fm$. Recall that two ideals $\mathfrak{a}$ and $\mathfrak{b}$ coprime to $\fm$ define the same class in $C_\fm$ if $\mathfrak{a}=\mathfrak{b}c$ for some $c\in K^\times $ such that
\begin{align}\label{eq:ray class}
v_\fq(c-1)\geq v_\fq(\fm) \text{ for all }\fq\mid \fm.  
\end{align}

 Class field theory identifies $C_\fm$ with the Galois group of the ray class field of $M$ modulo $\fm$. Then, by the Chebotaryov density theorem there exist infinitely many primes $\fp$ lying in the class of $\pi_0\cO_M$. Since $M$ has class number $1$ and $\cO_M^\times =\{\pm 1\}$, there exist a generator $\pi$ of $\fp$  and $c\in K^\times$ satisfying \eqref{eq:ray class} such that $\pi=\pi_0c$. But then $v_\fq(\frac{\pi}{\pi_0}-1)\geq v_\fq(\fm)$ for all $\fq\mid\fm$, as we wanted to show.

Having proved that $P_{a,N}$ is infinite, we can complete the proof of the Proposition as follows. If $E[N]$ and $E'[N]$ were isomorphic as $(\Z/N\Z)[G_\Q]$-modules, then we would have that
$$
a_p(E)\equiv a_p(E') \pmod N\,,
$$
for all but for a finite number of rational primes $p$.
But together with Lemma \ref{lemma: appm1} this implies that
$$
0=a_p(E')\equiv a_p(E)=\pm a \pmod N\,,
$$
for all but for a finite number of primes $p$ in $P_{a,N}$. This contradicts the fact that~$N$ does not divide $a$.
\end{proof}

\begin{corollary}\label{corollary: two primes}
Assume that $d\not =3,4$. Suppose that there exists a triple of natural numbers $(p_1,p_2,a)$, with $p_1,p_2\geq 5$ different primes, such that:
\begin{enumerate}[i)]
\item $p_i$ is split in $M/\Q$, inert in $M'/\Q$, and $\Tr_{M/\Q}(\pi_i)=a$, where $\pi_i\in \cO_M$ is such that $p_i=\pi_i\overline \pi_i$ for $i=1,2$. 
\item For every pair of elliptic curves $E$ and $E'$ defined over $\Q$ with CM by~$M$ and~$M'$, and for every prime divisor $q$ of $a$, the $(\Z/q\Z)[G_\Q]$-modules $E[q]$ and $E'[q]$ are not isomorphic.
\end{enumerate}
Then there is no genus $2$ curve $C$ defined over $\Q$ such that $\cX(C)\simeq M\times M'$.
\end{corollary}

\begin{proof}
 Assume for the sake of contradiction that there exists $C$ defined over $\Q$ such that $\cX(C)\simeq M\times M'$. By \cite[Prop. 4.5]{FKRS12}, there exist elliptic curves $E$ and $E'$ defined over $\Q$ with CM by $M$ and $M'$, respectively, such that $\Jac(C)\sim E\times E'$. By  Proposition~\ref{proposition: FreyKani}, there exists an integer $N\geq 2$ such that $E$ and $E'$ can be chosen so that $E[N]$ and $E'[N]$ are isomorphic as $(\Z/N\Z)[G_\Q]$-modules. Replacing $N$ by one of its prime divisors if necessary, we can assume that $N$ is prime. Applying Proposition \ref{proposition: just one prime} with $p_1$ and taking hypothesis $ii)$ into account, we obtain that $N=p_1$. One analogously obtains that $N=p_2$, which is a contradiction.
\end{proof}

\begin{table}[h]
\begin{center}
\begin{center}
\begin{tabular}{ |r r r r r| r r r r|  } 
 \hline
  $(d,d')$ & $p_1$ & $p_2$ & $a$  & &  $(d,d')$ & $p_1$ & $p_2$ & $a$\\
  \hline
$( 7 , 3 )$  &  29  &  113  &  2  & &   $( 43 , 19 )$  &  97  &  269  &  1 \\
\hline 
$( 8 , 3 )$  &  17  &  41  &  6  & &   $( 67 , 3 )$  &  17  &  419  &  1 \\
\hline 
$( 8 , 7 )$  &  73  &  1153  &  2  & &   $( 67 , 4 )$  &  151  &  419  &  1 \\
\hline 
$( 11 , 3 )$  &  5  &  71  &  3  & &   $( 67 , 7 )$  &  17  &  419  &  1 \\
\hline 
$( 11 , 4 )$  &  223  &  619  &  1  & &   $( 67 , 8 )$  &  151  &  821  &  1 \\
\hline 
$( 11 , 7 )$  &  223  &  619  &  1  & &   $( 67 , 11 )$  &  17  &  151  &  1 \\
\hline 
$( 11 , 8 )$  &  223  &  1213  &  1  & &   $( 67 , 19 )$  &  151  &  2027  &  1 \\
\hline 
$( 19 , 3 )$  &  5  &  233  &  1  & &   $( 67 , 43 )$  &  151  &  419  &  1 \\
\hline 
$( 19 , 4 )$  &  43  &  3463  &  1  & &   $( 163 , 3 )$  &  41  &  1019  &  1 \\
\hline 
$( 19 , 7 )$  &  5  &  1069  &  1  & &   $( 163 , 4 )$  &  367  &  1019  &  1 \\
\hline 
$( 19 , 8 )$  &  5  &  1069  &  1  & &   $( 163 , 7 )$  &  41  &  367  &  1 \\
\hline 
$( 19 , 11 )$  &  43  &  233  &  1  & &   $( 163 , 8 )$  &  367  &  1997  &  1 \\
\hline 
$( 43 , 3 )$  &  11  &  269  &  1  & &   $( 163 , 11 )$  &  41  &  1019  &  1 \\
\hline 
$( 43 , 4 )$  &  11  &  6719  &  1  & &   $( 163 , 19 )$  &  41  &  1019  &  1 \\
\hline 
$( 43 , 7 )$  &  97  &  269  &  1  & &   $( 163 , 43 )$  &  1019  &  1997  &  1 \\
\hline 
$( 43 , 8 )$  &  269  &  1301  &  1  & &   $( 163 , 67 )$  &  41  &  367  &  1 \\
\hline 
$( 43 , 11 )$  &  9041  &  10331  &  1  & &   \multicolumn{4}{c}{}    \\\cline{1-5} 

\end{tabular}
\end{center}
\vspace{6pt}
\caption{For each pair $(d,d')$, a triple $(p_1,p_2,a)$ satisfying the hypotheses of Corollary \ref{corollary: two primes}.}\label{table: triples}
\end{center}
\end{table}

In order to prove the main theorem of this section, we will apply Corollary~\ref{corollary: two primes} to each pair $(d,d')$ displayed in Table \ref{table: triples} using each of the displayed triples $(p_1,p_2,a)$. 
Before, we will describe the $3$-torsion fields of certain elliptic curves. This will be useful in order to verify that the triples $(p_1,p_2,a)$ displayed in Table \ref{table: triples} satisfy the hypotheses of Corollary \ref{corollary: two primes}. 

For an elliptic curve $E$ defined over $\Q$, let $\psi_{E,3}(x)$ denote its $3$-division polynomial, so that $\Q(x(E[3]))$ coincides with the splitting field of $\psi_{E,3}(x)$. Recall that if~$E$ is given by the Weierstrass equation
\begin{equation}\label{equation: WeierstrassEquation}
y^2=x^3+ax+b\,,\qquad \text{where }a,b\in \Q^\times,
\end{equation}
then $\psi_{E,3}(x)=3x^4+6ax^2+12bx-a^2$.
\begin{remark}\label{remark: IndependenceQuadraticTwist}
Let $E'$ be a quadratic twist of $E$. We may assume that $E'$ is given by the Weierstrass equation
$$
y^2=x^3+a/d^2x+b/d^3\,,\qquad \text{where }d\in \Q^\times.
$$ 
One trivially has $d^4\psi_{E',3}(x)=\psi_{E,3}(dx)$ and thus $\Q(x(E[3]))=\Q(x(E'[3]))$.
\end{remark}

\begin{lemma}\label{lemma: 8113torsion}
Let $E$ be an elliptic curve defined over $\Q$ with CM by $M=\Q(\sqrt{-d})$. 
\begin{enumerate}[i)]
\item If $d=8$ or $11$, then $\Q(E[3])$ is an extension of $M(\sqrt{-3})$ of degree $\leq 2$.
\item If $d=3$, then either:
\begin{enumerate}[a)]
\item $\Q(E[3])$ has degree divisible by $3$; or 
\item $E$ is of the form $y^2=x^3+b$, where $b\in \Q^\times$ is such that $4b\in \Q^{\times,3}$, and $\Q(E[3])=\Q(\sqrt{-3},\sqrt{b})$. 
\end{enumerate}
\end{enumerate}
\end{lemma}

\begin{proof}
As for part $i)$, by Remark \ref{rmk: referee} below, it will suffice to show that $\Q(x(E[3]))=M(\sqrt{-3})$. If $d=8$, then~$E$ is a quadratic twist of the elliptic curve $E_0$ attached to the $j$-invariant $2^6\cdot 5^3$ on \cite[page 483]{Sil94}. By Remark \ref{remark: IndependenceQuadraticTwist}, the field $\Q(x(E[3]))$ is the splitting field of
$$
\psi_{E_0,3}=(x^2 - 48x + 864)(x^2 + 48x - 7200)\,,
$$
which is $\Q(\sqrt{-2},\sqrt{-3})$. If $d=11$, then~$E$ is a quadratic twist of the elliptic curve $E_0$ attached to the $j$-invariant $-2^{15}$ on \cite[page 483]{Sil94}. By Remark \ref{remark: IndependenceQuadraticTwist}, the field $\Q(x(E[3]))$ is the splitting field of
$$
\psi_{E_0,3}=(x^2 - 132x + 4752)(x^2 + 132x - 6336)\,,
$$
which is $\Q(\sqrt{-11},\sqrt{-3})$.


We now turn to part $ii)$. By \cite[page 483]{Sil94}, $E$ pertains to one of the $\Qbar$-isomorphism classes corresponding to the $j$-invariants $0$, $2^4\cdot 3^3\cdot 5^3$ and $-2^{15}\cdot 3\cdot 5^3$. If $j(E)=2^4\cdot 3^3\cdot 5^3$, then we see as in the previous case that $\Q(x(E[3]))$ is the splitting field of
$$
(x - 3)(x^3 + 3x^2 - 21x + 25)\,,
$$
whose degree is divisible by $3$. If $j(E)=-2^{15}\cdot 3\cdot 5^3$, then $\Q(x(E[3]))$ is the splitting field of
$$
(x - 108)(x^3 + 108x^2 - 66096x + 4665600)\,,
$$
whose degree is divisible by $3$. Finally, suppose that $j(E)=0$. Then $E$ is given by \eqref{equation: WeierstrassEquation} with $a=0$, and $\Q(x(E[3]))$ is the splitting field of $x^3+4b$, which either has degree divisible by $3$ or $-4b \in \Q^{\times,3}$. In the latter case, let $c:=-(b/2)^{1/3}\in\Q^\times$. A straightforward calculation shows that $E[3]$ is generated by the points $P:=(2c,\sqrt{-3b})$ and $Q:=(0,\sqrt b)$. It follows that $\Q(E[3])=\Q(\sqrt{-3},\sqrt b)$.
\end{proof}

\begin{remark}\label{rmk: referee} For any elliptic curve $E$ over $\Q$ and $m\geq 3$ we have that \[[\Q(E[m])\colon \Q(x(E[m]))]\leq 2,\]
because \[\Gal(\Q(E[m])/\Q(x(E[m])))\simeq \overline{\rho}_{E,m}(G_{\Q(x(E[m]))})\subseteq \{\pm 1_{E[m]}\}.\]
To see the last inclusion, let $P$ and $Q$ denote a basis of $E[m]$, and let $\sigma$ be an element of $ \Gal(\Q(E[m])/\Q(x(E[m])))$. Then $\sigma(P) = \varepsilon\cdot P$, $\sigma(Q)=\delta\cdot Q$, and $\sigma(P+Q)=\gamma\cdot (P+Q)$ for some $\varepsilon, \delta, \gamma\in\{\pm 1\}$. Since $\sigma(P+Q)=\sigma(P)+\sigma(Q)$, we have
  \begin{equation*}
  (\gamma-\varepsilon)\cdot P +(\delta - \gamma)\cdot Q =0.  
\end{equation*}
Since $P$ and $Q$ are a basis and $m\geq 3$, this implies $\gamma = \varepsilon = \delta$. This means that  $\overline{\rho}_{E,m}(G_{\Q(x(E[m]))})\subseteq \{\pm 1_{E[m]}\}$, as desired.
\end{remark}

\begin{theorem}\label{theorem: three pairs}
Let $C$ be a genus 2 curve defined over $\Q$ such that $\cX(C)\simeq M\times M'$. Then 
$$
\cX(C)\simeq \Q(\sqrt {-d})\times \Q(\sqrt{-d'})\,,
$$
where $(d,d')\in \{ (4,3), (7,4), (8,4)\}$. 
\end{theorem}

\begin{proof}
Without loss of generality we can assume that $d>d'$. It is a straightforward computation to verify that all triples displayed on Table \ref{table: triples} satisfy hypothesis $i)$ of Corollary \ref{corollary: two primes}. The reader can find a SageMath code that verifies this at the file \texttt{find\_primes\_dirichlet.sage} of \cite{FFG22}. When $a=1$, hypothesis~$ii)$ is vacuous, from which we immediately deduce that
$$
(d,d')\in \{(4,3), (7,3), (7,4), (8,3), (8,4), (8,7), (11,3)\}\,.
$$

When $a>1$ to rule out a pair $(d,d')$, by Corollary \ref{corollary: two primes} it is enough to check that there do not exist elliptic curves $E$ and $E'$ over $\Q$ such that $E$ has CM  by $M=\Q(\sqrt{-d})$, $E'$ has CM by $M'=\Q(\sqrt{-d'})$, and a prime $q\mid a$ such that $E[q]\simeq E'[q]$ as $(\Z/q\Z)[G_\Q]$-modules. In fact, with the only exception of $d=8$, $d'=3$ and $q=3$, on can in fact see this with the stronger condition $\Q(E[q])\neq \Q(E'[q])$.

Case $(d,d')=(7,3)$, $q=2$.  On the one hand, by \cite[page 483]{Sil94},~$E$ pertains to one of the $\Qbar$-isomorphism classes corresponding to the $j$-invariants $-3^3\cdot 5^3$ and $3^3\cdot 5^3\cdot 17^3$. As $\Qbar$-isomorphic elliptic curves with these $j$-invariants are quadratic twists of each other, the field $\Q(E[2])$ is invariant in each of these two $\Qbar$-isomorphism classes (this can be seen by comparing Weierstrass equations). This yields two possibilities for $\Q(E[2])$: $\Q(\sqrt{-7})$ or $\Q(\sqrt{7})$. On the other hand, as recalled in the proof of Lemma \ref{lemma: 8113torsion}, $E'$ pertains to one of the $\Qbar$-isomorphism classes corresponding to the $j$-invariants $0$, $2^4\cdot 3^3\cdot 5^3$ and $-2^{15}\cdot 3\cdot 5^3$. Arguing as before we see that if $j(E')=2^4\cdot 3^3\cdot 5^3$, then $\Q(E'[2])=\Q(\sqrt 3)$ and that if $j(E')=-2^{15}\cdot 3\cdot 5^3$, then $\Q(E'[2])$ is the splitting field of $x^3 - 30x + 253/4$, which has degree divisible by~$3$. If $j(E')=0$, then $\Q(E'[2])$ is the splitting field of $x^3+b$, which is either $\Q(\sqrt{-3})$ or has degree divisible by $3$. In any case, we find that $\Q(E[2])$ and $\Q(E'[2])$ do not coincide.

Case $(d,d')=(8,7)$, $q=2$. As recalled in the proof of Lemma \ref{lemma: 8113torsion}, $E$ pertains to the $\Qbar$-isomorphism class of $j$-invariant $2^6\cdot 5^3$, which as previously argued implies that $\Q(E[2])=\Q(\sqrt{2})$. On the other hand, as we saw in the previous paragraph, we have that $\Q(E'[2])=\Q(\sqrt 7)$ or $\Q(\sqrt{-7})$. In particular, $\Q(E[2])$ and $\Q(E'[2])$ do not coincide.

Case $(d,d')=(11,3)$, $q=3$.  By Lemma \ref{lemma: 8113torsion}, the fields $\Q(E[3])$ and $\Q(E'[3])$ cannot coincide. Indeed, they clearly cannot coincide if $E'$ falls in case $a)$ of Lemma \ref{lemma: 8113torsion}; if it falls in case $b)$, then $\Q(E'[3])=\Q(\sqrt{-3},\sqrt{b}) $ for some $b\in\Q^\times $ such that $4b\in\Q^{\times,3}$ and if $\Q(E[3])=\Q(E'[3])$ from $i)$ of Lemma \ref{lemma: 8113torsion} we see that necessarily $\Q(E'[3])=\Q(\sqrt{-11},\sqrt{-3})$, a contradiction.

Case $(d,d')=(8,3)$, $q=2$.  We have already seen that $\Q(E[2])=\Q(\sqrt{2})$, and that $\Q(E'[2])$ is either $\Q(\sqrt{-3})$, $\Q(\sqrt{3})$, or has degree divisible by $3$.

Case $(d,d')=(8,3)$, $q=3$.  Suppose that $E[3]\simeq E'[3]$ as $(\Z/3\Z)[G_\Q]$-modules. A simple argument using Lemma~\ref{lemma: 8113torsion}, shows that if $\Q(E[3])$ and $\Q(E'[3])$ coincide, then $E'$ is of the form $y^2 = x^3 +b$ and   $b$ lies either in $-2\cdot \Q^ {\times,6}$ or in $2\cdot 3^ 3\cdot \Q^{\times,6}$. We may thus assume that~$E'$ is given by either 
$$
y^2=x^3-2\,, \qquad \text{or} \qquad y^2=x^3+54\,.
$$
In the first case, $P=(2,\sqrt 6)$ and $Q=(0,\sqrt {-2})$ constitute a basis of $E'[3]$. In the second case, a basis is given by $P=(-6,9\sqrt{-2})$ and $Q=(0,3\sqrt 6)$. Using this basis, it is easy to see that for any prime $p$ of good reduction for $E'$ split in $\Q(\sqrt{-3})$ and inert in $\Q(\sqrt{-2})$ we have 
$$
\Tr(\overline\varrho_{E',3}(\Frob_p))\equiv 1\pmod 3\,.
$$ 
This contradicts the fact that for any prime $p$ of good reduction for $E$ and inert in $\Q(\sqrt{-2})$ we have that $\Tr(\overline\varrho_{E',3}(\Frob_p))\equiv a_p(E)=0\pmod 3$.
\end{proof}

\begin{remark}
We note that for $(d,d')\in \{(4,3),(7,4),(8,4)\}$, there does not exist a triple $(p_1,p_2,a)$ satisfying the hypotheses of Corollary~\ref{corollary: two primes}. Obviously, this needs a justification only if $d=7$ or $8$. In that case, it is an elementary exercise to show that $\Tr_{M/\Q}(\pi)$ is even for every prime $\pi$ in $\cO_M$ of residue degree 1. But as we will see in the next section, there are choices of elliptic curves $E$ and $E'$ with CM by $M$ and~$M'$, respectively, for which $E[2]$ and $E'[2]$ are isomorphic as $\Z/2\Z[G_\Q]$-modules. These examples were found in \cite{LMFDB}. 
\end{remark}

\subsection{The three realizable pairs}\label{section: reapairs}
The elliptic curves 
$$
E_1\colon y^2 = x^{3} - 1\,,  \qquad E_1' \colon y^2 = x^{3} + 3 x\,.
$$
have CM by $\Q(\sqrt{-3})$ and by $\Q(\sqrt{-4})$ respectively. They both have a rational point of order $2$ and $\Q(E_1[2])=\Q(E_1'[2])=\Q(\sqrt{-3})$. Label the $2$-torsion points of $E_1$ and $E_1'$ as
$$
E_1[2] = \{O, P_1,P_2,P_3\}\,,\qquad E_1'[2] = \{O,Q_1,Q_2,Q_3\}  
$$
in such a way that $P_1\in E_1(\Q)$ and $Q_1\in E_1'(\Q)$. Then the map $\psi\colon E_1[2]\ra E_1'[2]$ given by $\psi(P_i)=Q_i$ is an isomorphism of $G_\Q$-modules that does not arise from an isomorphism between $E_1$ and $E_1'$ (since $j(E_1)\neq  j(E_1')$). Applying the method of Proposition \ref{proposition: gluing} gives the equation
\begin{align*}
  y^2 = -4251528x^6 + 6377292x^4 - 6377292x^2 + 2125764\,.
\end{align*}
The coefficients of the right hand side are divisible by the square factor $2^23^{12}=2125764 $, clearing them out we obtain the curve
\begin{align*}
  C_1\colon y^2 = -2x^6 + 3x^4 - 3x^2 + 1\,,
\end{align*}
whose Jacobian is isogenous over $\Q$ to $E_1\times E_1'$.

The same method can be applied to the curves
$$
 E_2\colon  y^2 = x^{3} + 7 x\,,\qquad  E_2' \colon y^2 =4x^3 - 3x^2 - 218x - 711\,.
$$
They have CM by $\Q(\sqrt{-4})$ and $\Q(\sqrt{-7})$, respectively, and $\Q(E_2[2])=\Q(E_2'[2])=\Q(\sqrt{-7})$. Similarly as before, $E_2[2]$ (resp. $E'_2[2]$) has a single nonzero rational point $P_1$ (resp. $Q_1$). Taking a map $\psi$ such that $\psi(P_1)=Q_1$ produces an equation which, after clearing a rational square factor, simplifies to the equation
\begin{align*}
  C_2\colon y^2 = 7500x^6 - 12375x^4 + 900x^2 - 480
\end{align*}
whose Jacobian is isogenous over $\Q$ to $E_2\times E_2'$.

Finally, the curves 
$$
E_3\colon y^2 = x^3 - 2x\,,\qquad 
E_3'\colon  y^2 = x^3 - x^2 - 3x - 1 
$$
have CM by $\Q(\sqrt{-4})$ and $\Q(\sqrt{-8})$, respectively, and $\Q(E_3[2])=\Q(E_3'[2])=\Q(\sqrt{2})$. The same method as above produces the curve
\begin{align*}
  C_3\colon y^2 = x^6 - x^4 - 3x^2 - 1
\end{align*}
whose Jacobian is isogenous over $\Q$ to $E_3\times E_3'$. All these calculations can be found and reproduced in the file \texttt{three\_realizable\_pairs.sage} from the repository \cite{FFG22}.

\section{Class number 4}\label{section: squares}

Throughout this section let $M=\Q(\sqrt{-d})$, where $d \in \Dm_{2,2}$. That is, $M$ is a quadratic imaginary field with class group $\cyc 2 \times \cyc 2$. Let $\Om_M$ denote the ring of integers of $M$ and let $H$ be the Hilbert class field of $M$. The results of this section complete the proof of Theorem \ref{theorem: main}. Recall the sets $\Dm_{2,2}^J \subseteq \Dm_{2,2}^S \subseteq \Dm_{2,2}$ defined by \eqref{equation: DNakamura} and \eqref{equation: DNarbonneRitzenthaler}.  In the course of \cite[Proof of Thm. 1.2; page 1416]{FG20} the following result is shown.

\begin{proposition}[\cite{FG20}]\label{proposition: FG20}
Suppose that $d\in \Dm_{2,2}$ and that $A$ is an abelian surface defined over $M$ such that $\cX(A)\simeq \M_2(M)$. Then $d\in \Dm_{2,2}^S$ and there is an elliptic curve $E_0$ defined over $H$ with CM by~$M$ and such that $A_H\sim E^2_0$.
\end{proposition}



In \S\ref{section: approach1} and \ref{section: approach2} we give two alternative methods to prove the following refinement.

\begin{proposition}\label{proposition: product iso}
Suppose that $d\in \Dm_{2,2}^S$ and that $A$ is an abelian surface defined over~$H$ such that there is an elliptic curve $E_0$ defined over $H$ with CM by~$\cO_M$ and such that $A\sim E_0^2$.
Then there exist elliptic curves $E$ and $E'$ defined over $H$ with CM by $\Om_M$ such that $A\simeq E\times E'$.
\end{proposition}

 Propositions \ref{proposition: FG20} and  \ref{proposition: product iso} allow us to work in  the setting of \cite{GHR19} and \cite{Nar22}, and in \S\ref{section: redstep3} we use their classification results to find the upper bound $\Dm_{2,2}^J$ in Theorem \ref{theorem: main}.
In \S\ref{section: realization}, we certify that the upper bound~$\Dm_{2,2}^J$ is sharp relying on the parametrizations of \cite{Kum15}.

\subsection{Proof of Proposition \ref{proposition: product iso}: first approach}\label{section: approach1} By Kani's theorem \cite[Thm. 2]{Kan11} applied to our particular setting, there exist $E$ and $E'$ defined over $H$ with CM by $M$ such that $A\simeq E\times E'$.  Suppose now that $E$ has CM by an order $\cO_f$ in $M$ of conductor $f$, and let us show that $f=1$. Let $H_f$ denote the ring class field of conductor $f$. Note that $H$ must contain $M(j(E))$, which by the theory of complex multiplication is $H_f$, and thus we deduce that $H$ and $H_f$ coincide. From the formula
$$
[H_f:H]=f \prod_{p\mid f}\left( 1-\left( \frac{-d}{p}\right) \frac{1}{p}\right)
$$
(see \cite[Cor. 7.24]{Cox89}, for example), we deduce that $H_f=H$ if and only if either $f=1$ or $f=2$ and $2$ is split in $M$. But one readily verifies that $2$ does not split in $\Q(\sqrt{-d})$ for any $d\in \Dm_{2,2}^S$, and therefore we must have $f=1$. The same argument shows that $E'$ has CM by $\cO_M$.

\subsection{Proof of Proposition \ref{proposition: product iso}: second approach}\label{section: approach2}

We will derive now the result from \cite{JKPRST18}.
Let $\Ic$ denote the category of torsion free finitely generated $\Om_M$-modules, and let $\Pc$ denote the category of abelian varieties defined over $H$ which are isogenous to a power of $E_0$.  Let 
$$
\mathcal Hom_{\Om_M}(-,E_0)\colon \Ic \rightarrow \Pc\,,\qquad \Hom(-,E_0)\colon \Pc\rightarrow \Ic 
$$ 
be the functors defined in loc. cit. By \cite[Thm. 7.5]{JKPRST18}, the abelian surfaces defined over $H$ which lie in the essential image of $\mathcal Hom_{\Om_M}(-,E_0)$ are those isomorphic to the product of two elliptic curves with CM by $\Om_M$. In order to complete the proof of Proposition \ref{proposition: product iso}, it will suffice to show that $\mathcal Hom_{\Om_M}(-,E_0)$ and $\Hom(-,E_0)$ are inverse equivalences of categories. This is verified in the proposition below.

\begin{proposition}
Suppose that $d\in \Dm_{2,2}^S$ and that $E_0$ is an elliptic curve defined over $H$ with CM by $\Om_M$. Then the functors $\mathcal Hom_{\Om_M}(-,E_0)$ and $\Hom(-,E_0)$ are inverse equivalences of categories.
\end{proposition}

\begin{proof}
Recall that $\overline \varrho_{E_0,\ell}$ denotes the representation obtained from the action of $G_H$ on $E_0[\ell]$.
By \cite[Thm. 7.7]{JKPRST18}, the functors $\mathcal Hom_{\Om_M}(-,E_0)$ and $\Hom(-,E_0)$ are inverse equivalences of categories if and only if for every prime $\ell$ the image of $\overline \varrho_{E_0,\ell}$ is not contained in the group of scalar matrices. It is well known that $\overline \varrho_{E_0,\ell}$ factors as
$$
\overline \varrho_{E_0,\ell} \colon G_H \rightarrow (\Om_M/\ell \Om_M)^{\times}\subseteq \Aut(E[\ell])\,. 
$$
In fact, by a result of Stevenhagen \cite{Ste01} (see also \cite[Thm. 1.4]{BC}), the image of $\overline \varrho_{E_0,\ell}$ in $(\Om_M/\ell \Om_M)^\times/\{\pm 1\}$ is surjective. This implies that
$$
\#\im(\overline \varrho_{E_0,\ell})\geq \frac{\#(\Om_M/\ell\Om_M)^{\times}}{2}=\begin{cases}
(\ell-1)(\ell-1)/2 & \text{ if $\ell$ is split in $M$,}\\
(\ell^2-1)/2 & \text{if $\ell$ is inert in $M$,}\\
\ell(\ell-1)/2 & \text{if $\ell$ ramifies in $M$.}\\
\end{cases}
$$
One readily verifies that if $\ell >3$, then the right-hand side of the above inequality is $>\ell-1$ in each of the three cases, and therefore for $\ell >3$ the image of $\overline \varrho_{E_0,\ell}$ cannot be contained in the group of scalar matrices. In order to verify that the action of $G_H$ on $E_0[\ell]$ for $\ell=2,3$ is not by scalars, it will suffice to verify that $[H(x(E_0[\ell]):H]>\ell-1$, as we have
$$
\#\im(\overline \varrho_{E_0,\ell})=[H(E_0[\ell]):H]\geq [H(x(E_0[\ell])):H]\,.
$$
The degree $[H(E_0[\ell]):H]$ depends on the $H$-isomorphism class of $E_0$, but the degree $[H(x(E_0[\ell])):H]$ only depends on the $\Qbar$-isomorphism class of $E_0$, as it coincides with the degree over $H$ of the splitting field of the $\ell$-division polynomial $\psi_{E_0,\ell}(x)$ of $E_0$, which is invariant under quadratic twist as we have already observed. For each $E_0$ representing one of the 72 $\Qbar$-isomorphism classes of elliptic curves with CM by $\Om_M$ for $d\in \Dm_{2,2}^S$, one readily verifies that the degree of the splitting field of $\psi_{E_0,\ell}(x)$ over $H$ is $>\ell -1$ for $\ell=2,3$. This completes the proof of the proposition.  
\end{proof}

\subsection{Narbonne's classification}\label{section: redstep3}
Let $d\in \Dm_{2,2}^S$ and $M=\Q(\sqrt{-d})$. Gélin, Howe, and Ritzenthaler \cite{GHR19} have determined the finite set $\mathcal P(d)$ of isomorphism classes of indecomposable principally polarized abelian surfaces $A$ defined over $\Qbar$ with field of moduli $\Q$ such that $A\simeq E^2$, where $E$ is an elliptic curve defined over $\Qbar$ with CM by~$\Om_M$. Extending their work, Narbonne \cite{Nar22} has recently determined the finite set $\Qm(d)\supseteq \Pm(d)$ of isomorphism classes of indecomposable principally polarized abelian surfaces $A$ defined over $\Qbar$ with field of moduli $\Q$ such that $A\simeq E\times E'$, where $E$ and $E'$ are elliptic curves defined over $\Qbar$ with CM by~$\Om_M$. The cardinalities of  $\Pm(d)$ and  $\Qm(d)$ may be found in \cite[Table~2]{GHR19} and \cite[Table 1]{Nar22}, and the corresponding polarizations are available on the webpages of the authors. 

Let $\Pm^*(d)$ (resp. $\Qm^*(d)$) denote the subset of $\Pm(d)$ (resp. $\Qm(d)$) consisting of $\Qbar$-isomorphism classes of Jacobians of genus two curves defined over~$\Q$.  The main goal of this section is to prove Proposition \ref{proposition: laststep} below. The key technical step is to show that $\Qm^*(280)=\emptyset$. The proof of this fact involves the study of the curves $C_i$  over $\Qbar$ associated to the Igusa--Clebsch invariants given in Table \ref{table: invariants 280}. 



\begin{table}[h!]
  \begin{center}

\def\arraystretch{2}
\begin{tabular}{ll}
  Curve & $(I_2,I_4,I_6,I_{10})$  \\ \hline
  \noalign{\vskip 2mm}    
     $C_1$ &\begin{small} $\begin{aligned} &\left(2^3, \frac{2^6 3^4 5^3 23^5 223}{29^2 383^2 1783^2},  \frac{2^9 3^3 5^5 23^5 31 \cdot 4613569}{29^3 383^3 1783^3},\frac{2^{14}3^7 5^6 23^{12}}{29^5 383^5 1783^ 5}\right) \end{aligned}$\end{small}
                                  \\ \noalign{\vskip 2mm}    \hline 
\noalign{\vskip 2mm}    
$C_2$ &\begin{small} $\begin{aligned} & \left( 2^3,  \frac{2^{10}3^4 7^2 11^6 23^ 4}{29^2 13537009^2},    \frac{2^{13}3^3 7^2 11^6 19 23^4 353\cdot 1180303}{5^2 29^3 13537009^3}, \frac{2^{26} 3^7 7^6 11^{12}23^{12}}{5^5 29^5 13537009^5}\right)\end{aligned}$\end{small}
                                   \\ \noalign{\vskip 2mm}     \hline

  \noalign{\vskip 2mm}    
$C_3$ &\begin{small} $\begin{aligned} & \left( 2^3, \frac{2^4 3^4 5^3 7^2 11^4 23^4 59}{5431080691^2} , -\frac{2^6 3^3 5^4 7^2 11^4 23^4 47 4621\cdot 60923}{5431080691^3}  ,\ \frac{2^8 3^7 5^6 7^6 11^{12}23^{12}}{5431080691^5}  \right)\end{aligned}$\end{small}
                                   \\ \noalign{\vskip 2mm}     \hline

  \noalign{\vskip 2mm}    
$C_4$ &\begin{small} $\begin{aligned} & \left( 2^3, \frac{2^4 3^4 5^3 11^ 5 23^4 101}{109^2 4392541^2} ,  \frac{2^6 3^3 5^3 11^5 19 23^4 107\cdot 30845219}{109^3 4392541^3} ,-\frac{2^8 3^7 5^6 11^{12}23^{12}}{109^5 4392541^ 5}  \right)\end{aligned}$\end{small}
                                   \\ \noalign{\vskip 2mm}     \hline

\end{tabular}
\end{center}
    \caption{Igusa--Clebsch invariants $(I_2,I_4, I_6, I_{10})$ of the genus 2 curves $C_i$. Each $C_i$ is nonsingular since $I_{10}\neq 0$.}\label{table: invariants 280}
\end{table}

\begin{proposition}\label{prop: jac of ci}
  For $i=1,\dots,4$, we have that $\Jac(C_i)\sim E_i^ 2$, where~$E_i$ is an elliptic curve defined over $\Qbar$ with CM by $\Q(\sqrt{-280})$.  
\end{proposition}
The proof of the above proposition is deferred to \S\ref{sec:280}, since it involves the techniques introduced in \S\ref{section: realization} which are independent of the methods of this section.
\begin{proposition}\label{prop: Q of 280}
The curves $C_i$ admit a model over $\Q(\sqrt{-280})$, but not over $\Q$. Moreover, we have that  $\mathcal Q(280)=\{\Jac(C_1), \Jac(C_2), \Jac(C_3), \Jac(C_4) \}$.
\end{proposition}

\begin{proof}
  From \cite[Table 1]{Nar22}, we know that $\#\mathcal Q(280)=4$.  Since the Igusa--Clebsch invariants of the $C_i$ are rational and pairwise distinct (as points in the weighted projective space), we see that the field of moduli of the $C_i$ is~$\Q$ and that the~$C_i$ are pairwise non $\Qbar$-isomorphic. Let $i\in \{1,\dots, 4 \}$. Using the formulas of \cite[p. 319]{Mes91}, we computed the Clebsch invariants of $C_i$, denoted as $A,B,C,D$ in loc. cit. (note that Mestre writes $A', B',C',D'$ for our $I_2,I_4,I_6,I_{10}$). Using Bolza's characterization of genus 2 curves with extra automorphisms  we checked that $\Aut(C_i)\simeq \Z/2\Z$. In the notation of  \cite[p. 318]{Mes91}, this amounts to checking that $R\neq 0$ and that at least one of $A,B,C$ is nonzero.  Since $I_2\neq 0$, we are in case 1 of \cite[p. 332]{Mes91}. Computing the quantities $x,y,z, d_2, T$ in loc. cit. from $A,B,C,D$, one sees that $zd_2\neq 0$ and the Hilbert invariant of the Mestre conic is given by $(d_2,-2zT)$ (cf. \cite[p. 333]{Mes91}). All these calculations are straightforward, but the numbers involved are too large to be displayed here; they are available in the file \cite[\texttt{curves\_280.sage}]{FFG22}. The result of the computation turns out to be the same for each $C_i$: the Hilbert invariant corresponds to the quaternion algebra over $\Q$ of discriminant 14. As explained in \cite[p. 333]{Mes91}, the curve $C_i$ can be defined over a field $F$ if, and only if, the Hilbert invariant is trivial over $F$. This is obviously not the case for $F = \Q$, but occurs for $F=\Q(\sqrt{-280})$ since $2$ and $7$ ramify in $\Q(\sqrt{-280})$. Hence, we may assume that the~$C_i$ are defined over $\Q(\sqrt{-280})$, which we do from now on.

   By Propositions~\ref{proposition: FG20},~\ref{proposition: product iso}, and~\ref{prop: jac of ci}, there exist elliptic curves $E_i$ and $E_i'$ defined over $H$ with CM by the maximal order of $\Q(\sqrt{-280})$ such that $\Jac(C_i)\simeq E_i\times E_i'$.   
Moreover, since the~$C_i$ are pairwise non $\Qbar$-isomorphic, Torelli's theorem implies that the $\Jac(C_i)$ are pairwise non $\Qbar$-isomorphic as principally polarized abelian surfaces. Putting the last two conclusions together, we infer that $\mathcal Q(-280)$ is precisely the set of the $\Jac(C_i)_{\Qbar}$. 
\end{proof}

The Igusa--Clebsch invariants of the curves $C_1,\dots,C_4$ were kindly provided to us by C. Ritzenthaler, who numerically computed them similarly as in \cite[\S4]{GHR19}.

\begin{proposition}\label{prop: empty 280}
 We have that $\Qm^*(280)=\emptyset$.  
\end{proposition}
\begin{proof}
  Suppose for the sake of contradiction that there is a genus 2 curve $C$ defined over $\Q$ such that $\Jac(C)_\Qbar\simeq E\times E'$, where $E$ and $E'$ are elliptic curves defined over $\Qbar$ with CM by the maximal order of $\Q(\sqrt{-280})$. Then, by Proposition \ref{prop: Q of 280}, we have that for one of the $C_i$ we have that $\Jac(C)_{\Qbar}\simeq \Jac(C_i)_{\Qbar}$, as principally polarized abelian surfaces. Torelli's theorem implies that $C_\Qbar \simeq C_{i,\Qbar}$, and hence $C$ and $C_i$ share the same Igusa--Clebsch invariants. This is a contradiction with the fact that Mestre's obstruction over $\Q$ is trivial for $C$, but nontrivial for $C_i$ by Proposition \ref{prop: Q of 280}.
\end{proof}
\begin{proposition}\label{proposition: laststep}
If $d\not\in \Dm_{2,2}^J$, then there does not exist a genus $2$ curve $C$ defined over $\Q$ such that 
$$
\Jac(C)_\Qbar\simeq E\times E'\,,
$$
where $E$ and $E'$ are elliptic curves defined over $\Qbar$ with CM by $\Om_M$.
\end{proposition}

\begin{proof}
We can assume that $d\in \Dm_{2,2}^S\setminus \Dm_{2,2}^J=\{280,  483,  595, 627, 795\}$ by Proposition \ref{proposition: FG20}. By Proposition \ref{prop: empty 280} we have that  $\Qm^*(280)=\emptyset$. 
As for the remaining discriminants, we read from \cite[Table 1]{Nar22} that 
$$
\Qm(483)=\Qm(627)=\emptyset\,,\qquad \Qm(595)=\Pm(595)\,,\qquad \Qm(795)=\Pm(795)\,.
$$
It follows from \cite[Table 4]{GHR19} that  $\Pm^*(595)$ and $\Pm^*(795)$ are empty.

\end{proof}

\begin{remark}
In the proof of the above proposition we ruled out the cases $d\in \Dm_{2,2}\setminus \Dm_{2,2}^S$ using Proposition \ref{proposition: FG20}. As a consistency check, we also confirmed the emptiness of $\Qm^*(d)$ for these values of $d$ using the strategy applied to the cases $d\in \Dm_{2,2}^S\setminus \Dm_{2,2}^J$ in the above proposition. 
\end{remark}

Proposition \ref{proposition: laststep} completes the proof of the part of Theorem \ref{theorem: main} exhibiting an upper bound for $\#\mathcal{B}$. In Table \ref{table:M2 of class number 2 2} we display curves showing that this upper bound is sharp. These curves were provided to us by Elkies and by Narbonne--Ritzenthaler (see Remark \ref{rmk: attributions}) as putative candidates obtained by numerical methods. Certifying correctness of these curves is the content of \S\ref{section: realization} below.
\subsection{Certifying the correctness of Table \ref{table:M2 of class number 2 2} }\label{section: realization} We first address the problem of showing that the curves in Table \ref{table:M2 of class number 2 2} have geometrically split Jacobian.  There is a vast literature on the problem of parametrizing pairs $(C,E)$, where $C$ is a genus 2 curve and~$E$ is an elliptic curve with the property that there is an optimal map $\phi\colon C\ra E$ of fixed degree $n>1$. We write $\tilde{\mathcal L}_n$ for the moduli space of such pairs, whose image in the moduli space of genus 2 curves is an open part of the Humbert surface of discriminant $n^2$ (cf. \cite{Hum}). The first parametrizations of such pairs were given by Legendre and Jacobi for $n=2$, and to Hermite, Goursat and Bolza among others for $n=3$ and $4$. This topic has been revisited more recently by Kuhn for $n=3$ (see \cite{Kuh88}), and by Shaska and his collaborators for $2\leq n\leq 5$ (see \cite{Sha04}, \cite{MSV09}, and \cite{SWWW08}). See also \cite{AP06} for a survey of results about the image of $\tilde{\mathcal{L}}_n$ on the moduli space of genus 2 curves. The connection of this problem to split Jacobians is that for such $C$ and $E$, the Jacobian $\Jac(C)$ is isogenous to the product of elliptic curves $E\times E'$, where $E'=\Jac(C)/\phi^*(E)$.

In this section we will use the more recent approach by Kumar \cite{Kum15}, who computes $\tilde{\mathcal L}_n$ for $2\leq n\leq 11$ by exploiting the fact that it is birrational to the Hilbert modular surface $Y_{-}(n^2)$ that parametrizes pairs $(A,\iota)$ of principally polarized abelian surfaces with an embedding $\iota\colon\Om_{n^2}\hookrightarrow \End(A)$, where $\Om_{n^2}$ denotes the quadratic order of discriminant $n^2$.

For each such $n=2,\dots,11$, \cite{Kum15} provides a birrational model for $\tilde{\mathcal L}_n$ of the form
\[
	z^2 = f(r,s),
\]
along with expressions of the Igusa--Clebsch invariants of the curve $C$ corresponding in the moduli space with a point $(C,E)$ with coordinates $(z,r,s)$. We note that these invariants do not depend on the $z$ coordinate, since the two points with the same $z$ coordinate correspond to points on the Hilbert modular surface having the same abelian surface and conjugate embeddings. In addition, \cite{Kum15} also provides the $j$-invariants associated to the two quotient curves $E$ and $E'$. For example, for $n= 3$, a birrational model for $\tilde{\mathcal L}_3$ is
\begin{align*}
  z^2 = 11664r^2-8(54s^3+27s^2-72s+23)r+ (s-1)^4(2s-1)^2.
\end{align*}
The Igusa-Clebsch invariants of a curve corresponding to a point $(C,E)$ with coordinates $r$ and $s$ are:

\begin{equation} \label{igcl}
\begin{aligned}
I_2 =&\ 8 s^{2} + 32 s - 16, \\
I_4 =&\ 4 s^{4} - 16 s^{3} - 192 r s + 24 s^{2} + 768 r - 16 s + 4 , \\
I_6 =&\  8 s^{6} + 16 s^{5} - 320 r s^{3}- 168 s^{4} + 1152 r s^{2} + 352 s^{3}\\ &- 2304 r^{2} + 6336 r s - 328 s^{2} - 2560 r + 144 s - 24, \\
I_{10} =&\  2^{14} r^3 .
\end{aligned}
\end{equation}
If we denote by $j_1$ and $j_2$ the $j$-invariants of $E$ and $E'$, then we have the following expressions for their product and sum:
\begin{equation} \label{j1j2}
\begin{aligned}
  j_1+j_2 =&\ (2 s^{9} - 17 s^{8} - 324 r s^{6} + 64 s^{7} + 1350 r s^{5}  - 140 s^{6}
  \\ &+ 17496 r^{2} s^{3} - 2097 r s^{4}+ 196 s^{5} - 23328 r^{2} s^{2}  \\  &
  + 1368 r s^{3} - 182 s^{4} - 314928 r^{3} + 9720 r^{2} s- 162 r s^{2}
   \\  &+ 112 s^{3} - 432 r^{2} - 198 r s - 44 s^{2} + 63 r + 10 s - 1)/{r^{2}},\\
j_1\cdot j_2 =&\  (s^{4} - 4 s^{3} + 432 r s + 6 s^{2} - 288 r - 4 s + 1)^3/{r^{3}},
\end{aligned}
\end{equation}
from which one can easily compute $j_1$ and $j_2$.

We claim that for every $d\in\Dm_{2,2}^J$ the Jacobian of the curve $C_d$ given in Table~\ref{table:M2 of class number 2 2}  has geometric endomorphism algebra isomorphic to $\M_2(\Q(\sqrt{-d}))$. 

We spell out the details of the proof of this claim for $d=84$, the smallest value in $\Dm_{2,2}^J$. In this case, one checks that the values:
\begin{align*}
  r &=\frac{728973}{2048} g^{2} - \frac{80180415}{8}, \\
  s &= -\frac{5}{1024} g^{2} + \frac{531}{4}, 
\end{align*}
where $g$ is a root of $x^{4} - 28160 x^{2} + 65536$, produce in \eqref{igcl} Igusa-Clebsch invariants that are equal (in the weighted projective space) to the Igusa-Clebsch invariants of $C_d$. This implies that there is a degree 3 optimal map from $C_d$ to an elliptic curve $ E$, and that  $\Jac(C_d)$ is geometrically isogenous to a product of elliptic curves $E\times E'$. We can compute the $j$-invariants of $E$ and $E'$ using \eqref{j1j2}, and one can check that the resulting $j_1$ and $j_2$ turn out to be roots of the Hilbert class polynomial of discriminant $-84$. That is, $E$ and $E'$ have CM by (the maximal order of) $\Q(\sqrt{-84})$. Since elliptic curves with CM by the same field are geometrically isogenous, we obtain that $\Jac(C_d)$ is geometrically isogenous to $E^2$, as we aimed to see.

We have performed analogous computations for all the remaining curves in Table~\ref{table:M2 of class number 2 2}. Equations for $\tilde{\mathcal{L}}_n$, the Igusa--Clebsch invariants of the curves in the tautological family of genus 2 curves above $\tilde{\mathcal{L}}_n$, and the $j$-invariants of the corresponding elliptic quotients for the rest of values $2\leq n\leq 11$ are available in  \cite{Kum15} (to be more precise, for $n=10,11$ they are not in the article but they are available in the source files of the arXiv submission of the article \url{https://arxiv.org/abs/1412.2849}). We display in the Table \ref{table:M2 of class number 2 2} the value of the degree $n$ of the optimal map in each case. For lack of space we do not display the values of $r$ and $s$ found for each curve, but they have been computed using the software SageMath and MAGMA. The reader can find the code (which can be executed and computes the values of $r$ and $s$ in each case and performs the verifications) at the file \texttt{curves\_all\_endomorphisms\_genus\_two.sage} of \cite{FFG22}. 

For the sake of completeness, and as a double check, the code also verifies the correctness of the entries of Tables  \ref{table:products of fields} and \ref{table:M2 of class number 2}. Observe that this is not strictly necessary, as the correctness of these curves either had been already established before in the literature, or it has been proved in the previous sections of the present article. 

The correctness of the entries of Tables  \ref{table:products of fields}, \ref{table:M2 of class number 2}, and \ref{table:M2 of class number 2 2} can be independently asserted using the general machinery of \cite{CMSV19}. Our choice to rely on the approach of \cite{Kum15} was motivated by the fact that this is specific to our situation of a geometrically split genus 2 Jacobian.   


\subsection{Proof of Proposition \ref{prop: jac of ci}} \label{sec:280}
To check that each $\Jac(C_i)$ is geometrically isogenous to a product $E_i\times E_i'$ of elliptic curves with CM by $\Q(\sqrt{-280})$, we proceed similarly as in \S \ref{section: realization}. That is to say, for each tuple of invariants $(I_2,I_4,I_6,I_{10})$ of Table \ref{table: invariants 280} we find a value of $n$ and algebraic values $r,s$ such that the point $(C,E)$ of $\tilde{\mathcal{L}}_n$ with coordinates $r$ and $s$ has the property that the Igusa--Clebsch invariants of $C$ are equal to $(I_2,I_4,I_6,I_{10})$. Then we compute the $j$-invariants of the elliptic quotients $E,E'$ of $\Jac(C)$ and we verify that they correspond to elliptic curves with CM by $\Q(\sqrt{-280})$.

For the curve $C_1$, we found the values $n=3$, and: 
\begin{align*}
  r &=  {2377704960008242957125}/{32}g^2 + {4149463607322291675}/{32}  , \\
  s &=  -{22754544369}/{16}g^2 + {39710559}/{16}   , 
\end{align*}
 where $g$ is a root of $3529707677063649x^4 - 6159886864926x^2 + 1$. That is, these values of $r$ and $s$ produce in \eqref{igcl} Igusa-Clebsch invariants that are equal  to the Igusa-Clebsch invariants of $C_1$. This implies that there is a degree 3 optimal map from $C_1$ to an elliptic curve $ E_1$, and that  $\Jac(C_1)$ is geometrically isogenous to a product of elliptic curves $E_1\times E_1'$. We can compute the $j$-invariants of $E_1$ and $E_1'$ using \eqref{j1j2}, and one can check that the resulting $j_1$ and $j_1'$ correspond to elliptic curves $E_1$ and $E_1'$ having CM by (the maximal order of) $\Q(\sqrt{-280})$. 

For $C_2$ the value of $n$ is 6 and the values for $r$ and $s$ are:
\begin{align*}
  r &=  157029304400/77864953147443g^2 - 1/207  , \\
  s &=  -245358288125/856514484621873g^2 + 7/36432   , 
\end{align*}
where $g$ is a root of the polynomial
\begin{footnotesize}
$$24658202440347859360000x^4 + 19019901335371323222343200x^2 + 141495739192361398214601.$$  
\end{footnotesize}
For $C_3$ the value of $n$ is 6 and the values for $r$ and $s$ are:
\begin{align*}
  r &=   126696250359648/1802429471005625g^2 - 3/460  , \\
  s &=    -37539629736192/3965344836212375g^2 + 7/2277  , 
\end{align*}
where $g$ is a root of the polynomial \begin{footnotesize}$$
                                         28536781964790410612171882496x^4 \allowbreak- 852547940586109864299395808000x^2\allowbreak + 245652324986738538567015625.$$\end{footnotesize}
For $C_4$ the value of $n$ is 8 and the values for $r$ and $s$ are:
\begin{align*}
  r &=   60998409763946496/2116393255955g^2 - 1654653000/6612799 , \\
  s &=   -5083200813662208/2116393255955g^2 + 151113348/6612799   , 
\end{align*}
where $g$ is a root of the polynomial
\begin{footnotesize}
$$179437017444556482342977273856x^4 - 1543467840241060392825323520x^2 + 28663778387478025.$$  
\end{footnotesize}
In all the cases one checks that the corresponding $j$-invariants are attached to elliptic curves with CM by $\Q(\sqrt{-280})$. Execution of the file \texttt{curves\_280.sage} from \cite{FFG22} can be used to verify the correctness of these computations.

\section{Explicit examples}\label{section: examples}
This section contains equations of genus 2 curves realizing each one of the possible 54 geometric endomorphism algebras of geometrically split Jacobians of curves defined over $\Q$.

\begin{table}[h!]
\begin{center}
\footnotesize
\setlength{\extrarowheight}{0.5pt}
\caption{Curves whose geometric endomorphism algebra is $\M_2(\Q)$, a product of fields, or $\M_2(\Q(\sqrt{-d}))$ with $d\in\Dm_1$.}\label{table:products of fields}
\vspace{6pt}
\begin{tabular}{lll}
  
Algebra & Curve & Source\\

  \hline\\
   $\M_2(\Q)$ & $y^2=x^6 + x^4 + x^2 + 1$ & \cite[Table 11, Row $E_1$]{FKRS12}\\
    \hline\\
 
  $\Q\times \Q$ & $y^2=x^6 + x^5 + x - 1$ & \cite[Table 11, Row $N(G_{3,3})$]{FKRS12}\\  \hline\\
  $\Q\times \Q(\sqrt{-3})$ & $y^2 = x^6 - 15x^2 + 22$ &  Proposition \ref{proposition: QxM}\\  \hline\\
  $\Q\times \Q(\sqrt{-4})$ & $y^2 = x^6 - 11x^2 + 14$ & Proposition \ref{proposition: QxM}\\  \hline\\
  $\Q\times \Q(\sqrt{-7})$ & $y^2 = x^6 - 35x^2 + 98$ & Proposition \ref{proposition: QxM}\\  \hline\\
  $\Q\times \Q(\sqrt{-8})$ & $y^2 = x^6 - 30x^2 + 56$ & Proposition \ref{proposition: QxM}\\  \hline\\
  $\Q\times \Q(\sqrt{-11})$ & $y^2 = x^6 - 264x^2 + 1694$ & Proposition \ref{proposition: QxM}\\  \hline\\
  $\Q\times \Q(\sqrt{-19})$ & $y^2 = x^6 -152 x^2 + 722$ &Proposition \ref{proposition: QxM}\\  \hline\\
  $\Q\times \Q(\sqrt{-43})$ & $y^2 = x^6 -3440 x^2 + 77658$ & Proposition \ref{proposition: QxM}\\  \hline\\
  $\Q\times \Q(\sqrt{-67})$ & $y^2  = x^6 -29480 x^2 + 1948226$ & Proposition \ref{proposition: QxM}\\  \hline\\
  $\Q\times \Q(\sqrt{-163})$ & $ y^2 = x^6 -8697680x^2 + 9873093538 $ & Proposition \ref{proposition: QxM}\\  \hline\\
  $\Q(\sqrt{-3})\times \Q(\sqrt{-4})$ & $y^2 = -2x^6 + 3x^4 - 3x^2 + 1$ &  Method of \cite[\S 3]{HLP00} (see \S\ref{section: reapairs})\\  \hline\\
  $\Q(\sqrt{-4})\times \Q(\sqrt{-7})$ & $ y^2 =7500x^6 - 12375x^4 + 900x^2 - 480$  & Method of \cite[\S 3]{HLP00} (see \S\ref{section: reapairs})\\  \hline\\
   $\Q(\sqrt{-4})\times \Q(\sqrt{-8})$& $y^2 = x^6 - x^4 - 3x^2 - 1$& Method of \cite[\S 3]{HLP00} (see \S\ref{section: reapairs})\\   \hline\\
  $\M_2(\Q(\sqrt{-3}))$ & $y^2=x^6+1$ & \cite[Table 11, Row $C_1$]{FKRS12}\\
    \hline\\
$\M_2(\Q(\sqrt{-4}))$ & $y^2=x^5+x^3+\frac{81}{196}x$ & \cite[p. 112]{Car01}\\
    \hline\\
$\M_2(\Q(\sqrt{-7}))$ & $y^2=x^5+x^3+\frac{3969}{16900}x$ & \cite[p. 112]{Car01}\\
    \hline\\
$\M_2(\Q(\sqrt{-8}))$ & $y^2=x^5-x$ & \cite[Table 11, Row $C_2$]{FKRS12}\\
    \hline\\
$\M_2(\Q(\sqrt{-11}))$ &  $y^2= 2 x^{6} + 11 x^{3} - 22 $ &  \cite[Table 4]{GHR19}\\
    \hline\\
$\M_2(\Q(\sqrt{-19}))$ & $\begin{aligned}y^2= & x^{6} + 1026 x^{5} + 627 x^{4} + 38988 x^{3} \\&- 11913 x^{2} + 370386 x - 6859   \end{aligned}$ & \cite[Table 4]{GHR19}\\
    \hline\\
$\M_2(\Q(\sqrt{-43}))$ & $\begin{aligned}y^2=&\  x^{6} + 48762 x^{5} + 1419 x^{4} + 4193532 x^{3}\\& - 61017 x^{2} + 90160938 x - 79507  \end{aligned}$ &\cite[Table 4]{GHR19}\\
    \hline\\
$\M_2(\Q(\sqrt{-67}))$  &  $\begin{aligned}y^2=&\   x^{6} + 785106 x^{5} + 2211 x^{4} + 105204204 x^{3}\\ &- 148137 x^{2} + 3524340834 x - 300763 \end{aligned}$ &  \cite[Table 4]{GHR19}\\
    \hline\\
$\M_2(\Q(\sqrt{-163}))$ & $\begin{aligned}y^2=&\  x^{6} + 1635420402 x^{5} + 5379 x^{4} + 533147051052 x^{3}\\ & - 876777 x^{2} + 43451484660738 x - 4330747   \end{aligned}$ & \cite[Table 4]{GHR19}\\ \hline
\end{tabular}
\end{center}
\end{table}
\newpage 
\begin{table}[h!]
\begin{center}
\footnotesize
\setlength{\extrarowheight}{0.5pt}
\caption{Curves whose geometric endomorphism algebra is $\M_2(\Q(\sqrt{-d}))$ with $d\in\Dm_2$ }\label{table:M2 of class number 2}
\vspace{6pt}
\begin{tabular}{lll}
  $d$   & Curve & Source\\
  \hline\\
 15    & $y^2=x^6+x^3+\frac{1}{20}$ &  \cite[p. 112]{Car01}\\  \hline\\
  20    & $y^2=x^5+ 5x^3+ 5x$ &  \cite[Table 4]{GHR19}\\  \hline\\
  24    &  $y^2= 3 x^{5} + 8 x^{3} + 6 x $ &\cite[Table 4]{GHR19}\\  \hline\\
  35    & $y^2= 83 x^6 - 184 x^5 + 405 x^4 + 80 x^3 - 255 x^2 + 136 x + 127$ &  Remark \ref{remark: classnumber2}\\  \hline\\
  40    &  $y^2= 9 x^{5} + 40 x^{3} + 45 x$ &  \cite[Table 4]{GHR19}\\  \hline\\
  51    & $y^2=x^6+x^3+\frac{4}{17}$ & \cite[p. 112]{Car01}\\  \hline\\
 52   &  $y^2= 9 x^{5} + 65 x^{3} + 117 x$ & \cite[Table 4]{GHR19} \\  \hline\\
  88    & $y^2= 99 x^{5} + 280 x^{3} + 198 x $ & \cite[Table 4]{GHR19}\\  \hline\\
  91    & $\begin{aligned}y^2=&\ 8477 x^6 + 11340 x^5 - 30513 x^4   \\&  -  15336 x^3  42135 x^2 - 11124 x + 1637 \end{aligned}$ & Remark \ref{remark: classnumber2}\\  \hline\\
  115    & $\begin{aligned}y^2=&\  2647 x^6 - 5736 x^5 + 7545 x^4 + 1520 x^3  \\&  - 4995 x^2 + 4824 x + 3483 \end{aligned}$ & Remark \ref{remark: classnumber2}\\  \hline\\
  123    & $y^2=x^6+x^3+\frac{256}{1025}$ & \cite[p. 112]{Car01}\\  \hline\\
  148    &$y^2= 441 x^{5} + 5365 x^{3} + 16317 x$  & \cite[Table 4]{GHR19}\\  \hline\\
  187    &$\begin{aligned}y^2=&\ 258093 x^6 + 519750 x^5 - 364518 x^4 -  1612000 x^3  \\&  + 286356 x^2 + 719000 x + 248616 \end{aligned}$&  Remark \ref{remark: classnumber2}\\  \hline\\
  232    & $y^2=9801 x^{5} + 105560 x^{3} + 284229 x $ & \cite[Table 4]{GHR19}\\  \hline\\
  235    & $\begin{aligned}y^2=&\ 63847 x^6 - 89256 x^5 - 89655 x^4 -   27280 x^3   \\&  + 135405 x^2 + 105624 x + 48843 \end{aligned}$ & Remark \ref{remark: classnumber2}\\  \hline\\
 267    &  $y^2=x^6+x^3+\frac{62500}{250001}$& \cite[p. 112]{Car01}\\  \hline\\
  403    & $\begin{aligned}y^2=&\ 19149677 x^6 - 28664274 x^5 + 11992662 x^4   \\& - 83358256 x^3  + 76850532 x^2 - 3824904 x + 53478712 \end{aligned}$ &  Remark \ref{remark: classnumber2}\\  \hline\\
  427    & $\begin{aligned}y^2=&\ 178773643 x^6 + 67725720 x^5 -  770837595 x^4  \\ & - 7030800 x^3  + 871950585 x^2 - 303837480 x + 27419167\end{aligned}$ & Remark \ref{remark: classnumber2}\\ \hline  
                                                             
\end{tabular}
\end{center}
\end{table}

\begin{table}
\begin{center}
\footnotesize
\setlength{\extrarowheight}{0.5pt}
\caption{Curves whose geometric endomorphism algebra is $\M_2(\Q(\sqrt{-d}))$ with $d\in\Dm_{2,2}^J$.  The column $n$ denotes the degree of an optimal map from the curve to an elliptic curve.}\label{table:M2 of class number 2 2}
\vspace{6pt}

\begin{tabular}{lll}
  
$d$ & $n$  & Curve\\ \hline\\
  84 & $3$  & $y^2=(2x^2-6x-9)(4x^4-24x^3-4x-3)$  \\ \hline\\
  120 & $6$  & $y^2=(2x^2-2x-1)(6x^2-20x+25)(18x^2+24x+5)$  \\ \hline\\
  132 & $5$  & $y^2 = (x^2 - 2x - 10) (121x^4 + 572x^3 + 176x^2 + 40x + 4)$\\ \hline\\
  168 & $8$  & $y^2=(2x^2+26x-3)(14x^2-16x+1)(14x^2+28x+23)$ \\ \hline\\
  228 & $5$  & $y^2=(x^2+2x-2)(15625x^4+62500x^3-5832x+2916)$  \\ \hline\\
  372 & $7$  & $\begin{aligned}y^2=&\ 1771561 x^{6} - 10629366 x^{5} + 10629366 x^{4} + 10797488 x^{3}\\& + 5062500 x^{2} + 10125000 x + 3375000\end{aligned}$  \\ \hline\\
  408 & $10$  & $\begin{aligned}y^2 =&\ 350354606 x^{6} - 590700831 x^{5} - 816200091 x^{4} + 1046797684 x^{3}\\& + 824978466 x^{2} - 502917081 x - 378640481\end{aligned}$ \\ \hline\\
  435 & $7$  & $\begin{aligned}y^2 =&\ 10466471x^6 + 11098203x^5 - 95645550x^4 - 106834905x^3\\&        + 95572050x^2 - 49559037x + 7237649\end{aligned}$ \\ \hline\\
  520 & $6$  & $y^2=(19x^2+54x+37)(319x^2-444x+171)(445083x^2+396482x-144363)$ \\ \hline\\
  532 & $3$  & $ y^2=(50251x^2 + 27624x - 17865) (29x^4 - 128x^3 + 114x^2 + 160x - 75)$ \\ \hline\\
  708 & $11$  & $y^2 = (6655x^2 - 3278x - 1966) (14225761x^4 + 9092252x^3 - 47672800x^2 + 48112648x + 6781636)$ \\ \hline\\
  1012 & $5$  & $y^2 = -(x^2 - 1350x + 2326)(8385623x^4 + 15948252x^3 - 44642960x^2 - 12762360x - 1510324)$ \\ \hline\\
  1435 & $11$ & $\begin{aligned} y^ 2  =&\ 539468307211x^6 + 8859839654637x^5 +  71566557542400x^4  + 152349804189405x^3 \\&- 44616942974400x^2 -  3819661212483x + 1228993875229\end{aligned}$\\
  \hline
 \end{tabular}
\end{center}
\vspace{0.6cm}
\begin{remark}\label{rmk: attributions}
  The curves corresponding to discriminants $-1012$ and $-1435$ were provided to us by Narbonne--Ritzenthaler (private communication). They were computed with the method described in \cite[\S4.1]{GHR19}: from the principally polarized abelian surface $(E_1\times E_2,\varphi)$ with field of moduli $\Q$ one computes approximations to the Cardona--Quer invariants by means of the theta constants; computing with enough precision allows us to identify the invariants as rational numbers, and then to compute a model over $\Q$ with these invariants. The rest of the curves that appear in the table were provided to us by Elkies (private communication). They were computed by specializing families of K3 surfaces parametrized by a Shimura curve  at CM points (cf. \cite{El98}, \cite{El08}, \cite{ES13}).
\end{remark}
\end{table}

\end{document}